\documentclass[a4paper,12pt]{article}

\usepackage{amsfonts,amsmath,amssymb,amsthm}
\usepackage{mathtools}
\usepackage{xcolor}
\usepackage{scalerel}
\usepackage{authblk}
\usepackage{enumitem}
\usepackage{soul}

\usepackage{tikz,tikz-cd,tikzit,quiver}
\usetikzlibrary{positioning}
\setlist[description]{font=\normalfont\itshape}

\newcounter{nthm}[section]
\theoremstyle{definition}
\newtheorem{definition}[nthm]{Definition}
\newtheorem{remark}[nthm]{Remark}
\newtheorem{notation}[nthm]{Notation}
\newtheorem{example}[nthm]{Example}
\theoremstyle{plain}
\newtheorem{proposition}[nthm]{Proposition}
\newtheorem{lemma}[nthm]{Lemma}
\newtheorem{theorem}[nthm]{Theorem}
\newtheorem{corollary}[nthm]{Corollary}

\newcommand{\ie}{\emph{i.e.\@}}
\newcommand{\eg}{\emph{e.g.\@}}

\newcommand{\ifc}{\text{if\quad}}
\newcommand{\undefined}{\text{undefined}}
\newcommand{\otherwise}{\text{otherwise}}

\newcommand{\im}{\mathrm{im}}

\newcommand{\into}{\hookrightarrow}
\newcommand{\pto}{\rightharpoonup}
\newcommand{\keq}{\simeq}
\newcommand{\fml}[1]{\ensuremath{\mathbf{#1}}}

\newcommand{\fpset}[1]{\mathcal{F}(#1)}
\newcommand{\Mid}{\ \middle|\ }

\newcommand{\Nset}{\mathbb{N}}

\newcommand{\Rset}{\mathbb{R}}

\DeclarePairedDelimiter{\absolute}{\lvert}{\rvert}
\newcommand{\abs}[1]{\absolute*{#1}}

\newcommand{\id}{\mathrm{id}}
\newcommand{\iso}{\cong}
\newcommand{\xto}[1]{\xrightarrow{#1}}
\newcommand{\comma}[2]{(#1 \!\Rightarrow\! #2)}
\newcommand{\Def}{\text{Def}}

\newcommand{\cat}[1]{\ensuremath{\mathbf{#1}}}
\newcommand{\SCat}[1]{\ensuremath{\mathbf{\Sigma}_\mathrm{#1}}}
\newcommand{\Set}{\cat{Set}}

\newcommand{\TopHaus}{\cat{Top}_{\mathrm{Haus}}}
\newcommand{\CMon}{\cat{CMon}}
\newcommand{\HausCMon}{\cat{HausCMon}}
\newcommand{\HAG}{\cat{HausAb}}

\title{Categories of sets with infinite addition}
\author[1,2]{Pablo Andr\'es-Mart\'inez\footnote{
The authors are grateful for the insightful comments and suggestions from the anonymous reviewer, which improved the proofs and fixed technical issues of previous versions of the manuscript.  
PAM was supported by the Engineering and Physical Sciences Research Council (grant EP/L01503X/1), EPSRC Centre for Doctoral Training in Pervasive Parallelism at the University of Edinburgh, School of Informatics.}}
\author[1]{Chris Heunen}
\affil[1]{\footnotesize The University of Edinburgh, School of Informatics, 10 Crichton Street, Edinburgh, UK}
\affil[2]{\footnotesize Quantinuum, Terrington House, 13-15 Hills Road, Cambridge, UK}

\begin{document}

\maketitle

\begin{abstract}
  We consider sets with infinite addition, called $\Sigma$-monoids, and contribute to their literature in three ways.
  First, our definition subsumes those from previous works and allows us to relate them in terms of adjuctions between their categories. In particular, we discuss $\Sigma$-monoids with additive inverses.
  Second, we show that every Hausdorff commutative monoid is a $\Sigma$-monoid, and that there is a free Hausdorff commutative monoid for each $\Sigma$-monoid.
  Third, we prove that $\Sigma$-monoids have well-defined tensor products, unlike topological abelian groups.
\end{abstract}

\section{Introduction}

A commutative monoid lets you add together any finite number of elements to produce an unambiguous result.
But if you want to aggregate an infinite collection of elements, monoids fall short, and we need a way to decide whether a given infinite collection can be assigned a well-defined result.
If the monoid is endowed with a topology, such a decision can be made in terms of the convergence of the sequence of partial sums --- this is how convergence of series in $\Rset$ is defined --- but simpler criteria than a topology exist (see~\cite{Haghverdi} for some examples).

Any such \emph{summability criterion} ultimately lifts binary addition to a partial function \(\Sigma\) that maps collections of elements of \(X\) to their result, if defined.
In this context, the criterion (\eg{} the topology) has been abstracted away, but it is crucial to axiomatise the structure it imprinted on \(\Sigma\), so that you can still reason about summability in this more general framework.
This has been the purpose of previous work introducing \(\Sigma\)-monoids~\cite{Haghverdi, Higgs}. Our work provides three novel contributions.
\begin{itemize}
  \item We classify and generalise the different flavours of \(\Sigma\)-monoids appearing in the literature, and describe their relation to each other in terms of adjuctions between their categories (Section~\ref{sec:SCat}). In particular, our most general \(\Sigma\)-monoid structure admits additive inverses (in contrast to~\cite{Haghverdi}) and generalises partially commutative monoids (in contrast to~\cite{Higgs}).
  \item We show that every Hausdorff commutative monoid is an instance of a \(\Sigma\)-monoid (Section~\ref{sec:topology}) and that the corresponding forgetful functor has a left adjoint. We interpret this as evidence that the key aspects of the structure imprinted on \(\Sigma\) by the topology are preserved.
  \item We show that each of the flavours of \(\Sigma\)-monoids considered here have a well-defined tensor product (Section~\ref{sec:tensor}).
\end{itemize}

The category of topological abelian groups lacks proper tensor products~\cite{TensorTAG, TensorMonoidalTAG}, but our results show that the tensor product of two Hausdorff commutative monoids is a \(\Sigma\)-monoid.
Thus we may enrich categories over \(\Sigma\)-monoids, where composition respects addition of morphisms, and hom-objects are Hausdorff commutative monoids.

This work originates from the effort to extend the categorical semantics of iteration from Manes and Arbib~\cite{ManesArbib} to quantum programs.
In quantum programs, it is theoretically possible to create linear combinations of different control flow paths, where the coefficients are complex numbers and, hence, admit additive inverses~\cite[Chapters 11 and 12]{Ying}.
Furthermore, parallel composition of states and programs in quantum computing is given by the tensor product~\cite{CQM}, contextualising the relevance of Section~\ref{sec:tensor}.
The generalised $\Sigma$-monoids presented in this paper are a cornerstone in the framework for categorical semantics of quantum while loops developed in the dissertation of the first author~\cite{pablo:thesis}.

\section{Notation}

We begin establishing the terminology and notation used throughout the paper.

\begin{definition} \label{def:families}
  A \emph{family} of elements of \(X\) is a multiset \(\{x_i \in X\}_{i \in I}\) indexed by a countable set \(I\).\footnote{Being an indexed multiset, there may be distinct indices \(i \neq j \in I\) with \(x_i = x_j\).}
  We consider two families \(\{x_i \in X\}_{i \in I}\) and \(\{y_j \in X\}_{j \in J}\) the same if there is a bijection \(\phi \colon I \to J\) satisfying $x_i=y_{\phi(i)}$ for all $i\in I$.
  Fix some infinite set $\mathcal{I}$ and let $X^*$ be the collection of all families of elements of \(X\) whose indexing set $I$ is a subset $I \subseteq \mathcal{I}$.
\end{definition}

\begin{proposition} \label{prop:family_set}
  For any set \(X\), the collection \(X^*\) is small.
\end{proposition} 
\begin{proof}
  Represent the multiset \(\{x_i \in X\}_{i \in I}\) as a set \(\{(i, x_i) \mid i \in I\} \subseteq \mathcal{I} \times X\). Since $\mathcal{I}$ and $X$ are sets, $2^{\mathcal{I} \times X}$ is small. It follows from the definition of $X^*$ that it is a quotient of a subset of $2^{\mathcal{I} \times X}$ and, hence, it is also small.
\end{proof}

\begin{notation} \label{not:families}
  We refer to the family \(\{x_i \in X\}_{i \in I}\) using the shorthand \(\{x_i\}_I\) where we use the convention that lower case letters \(x\), \(i\) are elements of the set denoted by the corresponding uppercase letter \(X\), \(I\).
  When the indexing set is not specified, all elements of the family are explicitly given within curly brackets: \(\{x\}\) is the singleton family and \(\{a,b,c\}\) is a finite family containing three elements.
  When explicit reference to the elements of the family is not needed, we denote a family with a bold font letter \(\fml{x} \in X^*\).
  Disjoint union of families is defined by the corresponding operations on their indexing sets:
  \[
    \{x_i\}_I \uplus \{x_j\}_J = \{x_k\}_{k \in I \uplus J}
  \]
  where we have assumed that $I$ and $J$ are disjoint, so that $I \uplus J \subseteq \mathcal{I}$. If this is not the case, recall that for any bijection \(\phi \colon J \to J'\) we consider $\{x_{j'}\}_{J'}$ and $\{x_{\phi(j)}\}_J$ to be the same family and, since $\mathcal{I}$ is an infinite set, we can assume without loss of generality that $I$ and $J'$ are disjoint.
  We do not define union of families, since we do not require it for the definitions and results in this paper.
  Intersection of families is defined below, and it only applies to the case where both families are restrictions of a larger one.
  If \(f \colon X \to Y\) is a function and \(\fml{x} = \{x_i\}_I\) a family in \(X^*\), we write \(f\fml{x}\) for the family \(\{f(x_i)\}_I\) in \(Y^*\).

\end{notation}

\begin{definition} \label{def:subfamily}
  A family \(\{x_j\}_J\) is a \emph{subfamily} of another \(\{x_i\}_I\) if there is an injection \(\phi \colon J \to I\) such that \(x_j = x_{\phi(j)}\) for each \(j \in J\).
  Formally, the injection $\phi$ is part of the data defining the subfamily, but in practice we always restrict to the case where $J \subseteq I$, so the injection is trivial.
  This gives rise to a partial order \((X^*, \subseteq)\) where \(\fml{x'} \subseteq \fml{x}\) if and only if \(\fml{x'}\) is a subfamily of \(\fml{x}\).
  If \(\fml{x}_0 \subseteq \fml{x}\) and \(\fml{x}_1 \subseteq \fml{x}\) we define the intersection $\fml{x}_0 \cap \fml{x}_1$ to be the subfamily obtained from the intersection of their indexing sets or, more formally, from the indexing set that is the pullback of the injections defining $\fml{x}_0$ and $\fml{x}_1$.
\end{definition}

\begin{notation} \label{not:keq}
  For partial functions \(f \colon A \pto B\) and \(g \colon A' \pto B\), the Kleene equality \(f(a) \keq g(a')\) indicates that \(f(a)\) is defined if and only if \(g(a')\) is defined and, when they are, their results agree.
  In particular, \(f(a) \keq b\) implies that \(f(a)\) is defined and is equal to \(b\).
  The equals sign \(f(a) = b\) is only used in the presence of partial functions when the context has already established that \(f(a)\) is defined.
\end{notation}

Section~\ref{sec:topology} will study the correspondence between \(\Sigma\)-monoids and Hausdorff commutative monoids. We assume basic knowledge of topology -- in particular, the definitions and properties of Hausdorff spaces and continuous functions -- the second chapter of~\cite{Munkres} contains an in-depth introduction. The following definitions are standard.

\begin{definition} \label{def:Hausdorff_commutative_monoid}
  Let \((X,\tau)\) be a Hausdorff space and let \((X,+)\) be a monoid.
  Then \((X,\tau,+)\) is a \emph{Hausdorff commutative monoid} if \(+ \colon X \times X \to X\) is continuous.
  A \emph{Hausdorff abelian group} is a Hausdorff commutative monoid \((X,\tau,+)\) where \((X,+)\) is an abelian group and the inversion map \(x \mapsto -x\) is continuous.
\end{definition}

The notion of convergence mentioned in the introduction is formalised in  general topology in terms of nets and their limit points.

\begin{definition} 
  A nonempty partially ordered set $(D,\leq)$ is \emph{directed} when every pair of elements \(a,b \in D\) has an upper bound $c \geq a,b$.
  A \emph{net} in $X$ is a function $\alpha \colon D \to X$ from a directed set $D$.
  A point \(x \in X\) is a \emph{limit point} of the net, written
  \begin{equation*}
    \lim \alpha = x\text,
  \end{equation*}
  if and only if for each open set \(U\) containing \(x\) there is an element \(a \in D\) such that \(\{\alpha(b) \mid a \leq b \in D \} \subseteq U\).
\end{definition}


\begin{definition} \label{def:subnet}
  A net \(\alpha \colon A \to X\) is a \emph{subnet} of \(\beta \colon B \to X\) if there is a function \(m \colon A \to B\) such that:
  \begin{itemize}
    \item the net \(\alpha\) factors through \(\beta\), that is, \(\alpha = \beta \circ m\);
    \item \(m\) is \emph{isotone}: if \(a \leq a'\) in $A$, then \(m(a) \leq m(a')\) in $B$;
    \item \(m(A)\) is \emph{cofinal} in \(B\): for every \(b \in B\) there is \(a \in A\) satisfying \(b \leq m(a)\).
  \end{itemize}
\end{definition}

\begin{proposition}\cite[Chapter 3, Exercise 8]{Munkres}\label{prop:subnet_limit}
  If $\alpha \colon A \to X$ is a subnet of $\beta \colon B \to X$ in a Hausdorff space $X$, then:
  \begin{equation*}
    \lim \beta = x \implies \lim \alpha = x
  \end{equation*}
\end{proposition}

\section{\(\Sigma\)-monoids}
\label{sec:SCat}

The notion of \emph{strong} \(\Sigma\)-monoid that will be presented in this section was originally proposed by Haghverdi~\cite{Haghverdi}.
However, strong \(\Sigma\)-monoids are too restrictive for many applications: Proposition~\ref{prop:no_inverses} below shows that they forbid additive inverses.
In this section we introduce a weaker version of \(\Sigma\)-monoids that does admit additive inverses.

\begin{definition} \label{def:wSm}
  Let \(X\) be a set and \(\Sigma \colon X^* \pto X\) be a partial function. Say \(\fml{x} \in X^*\) is a \emph{summable family} if \(\Sigma \fml{x}\) is defined. The pair \((X,\Sigma)\) is a \emph{weak $\Sigma$-monoid} if the following axioms are satisfied.
  \begin{description}
    \item[{Singleton.}] \(\Sigma \{x\} \keq x\) for all \(x \in X\).
    \item[{Neutral element.}] There is an element \(0 \in X\) such that \(\Sigma \emptyset \keq 0\), and
    if a family \(\fml{x} \in X^*\) is summable, then so is its subfamily \(\fml{x}_\emptyset\) of elements other than \(0\).
    \item[{Bracketing.}] If \(\{\fml{x}_j\}_J\) is an indexed set of \emph{finite} summable families \(\fml{x}_j \in X^*\), and \(\fml{x} = \uplus_J \fml{x}_j\), then:
    \begin{equation*} \label{eq:bracketing}
      \Sigma \fml{x} \keq x \implies \Sigma \{\Sigma \fml{x}_j\}_J \keq x
    \end{equation*}
    \item[{Flattening.}] If \(\{\fml{x}_j\}_J\) is an indexed set of summable families \(\fml{x}_j \in X^*\) whose indexing set \(J\) is \emph{finite}, and \(\fml{x} = \uplus_J \fml{x}_j\), then:
    \begin{equation*} \label{eq:fin_flattening}
      \Sigma \{\Sigma \fml{x}_j\}_J \keq x \implies \Sigma \fml{x} \keq x
    \end{equation*}
  \end{description}
\end{definition}

Bracketing and flattening together capture associativity and commutativity. Notice that these are not exactly dual to each other: bracketing assumes that each family \(\fml{x}_j\) is finite, whereas flattening admits infinite families \(\fml{x}_j\) but assumes there is only a finite number of them.
As Proposition~\ref{prop:no_inverses} will establish, the flattening axiom's requirement that \(J\) is finite is necessary for weak \(\Sigma\)-monoids to admit additive inverses.
On the other hand, the bracketing axiom's requirement that each family \(\fml{x}_j\) is finite is necessary so that every Hausdorff commutative monoid is a weak \(\Sigma\)-monoid whose summable families are precisely those that converge according to the topology (see Proposition~\ref{prop:wSm_Hausdorff} below).
The following proposition establishes that \(0 \in X\) acts as the neutral element of the \(\Sigma\)-monoid.

\begin{proposition} \label{prop:neutral_element}
  Let \((X,\Sigma)\) be a weak \(\Sigma\)-monoid with \(\Sigma \emptyset = 0\). Let \(\fml{x} \in X^*\) be a family such that \(\Sigma \fml{x} \keq x\) for some \(x \in X\).
  Let \(\fml{0} \in X^*\) be a (possibly infinite) family whose elements are all \(0\); then \(\Sigma (\fml{x} \uplus \fml{0}) \keq x\).
  Conversely, let \(\fml{x}_\emptyset \subseteq \fml{x}\) be the subfamily of elements other than \(0\); then \(\Sigma \fml{x}_\emptyset \keq x\).
\end{proposition} \begin{proof}
  Since \(\Sigma \emptyset = 0\), the family \(\fml{0}\) may be equivalently defined as
  \begin{equation*}
    \fml{0} = \{\Sigma \emptyset, \Sigma \emptyset, \ldots\}
  \end{equation*}
  and, due to bracketing, \(\Sigma \emptyset \keq 0 \implies \Sigma \fml{0} \keq 0\).
  Moreover, for any element \(x \in X\), the singleton axiom imposes that \(\Sigma \{x\} \keq x\) and, clearly, \(\{x\} = \{x\} \uplus \emptyset\). Thus, according to bracketing:
  \begin{equation*}
    \Sigma \{x\} \keq x \implies \Sigma \{\Sigma \{x\}, \Sigma \emptyset\} \keq x.
  \end{equation*}
  Using the singleton axiom and the definition of \(0\), the right hand side can be rewritten as \(\Sigma \{x,0\} \keq x\).
  It now follows that \(\Sigma \{\Sigma \fml{x}, \Sigma \fml{0}\} \keq x\), and flattening gives \(\Sigma (\fml{x} \uplus \fml{0}) \keq x\), proving the first part of the claim.

  On the other hand, we may partition any summable family \(\fml{x} \in X^*\) into its subfamily \(\fml{x}_\emptyset\) and its subfamily \(\fml{0}\) so that \(\fml{x} = \fml{x}_\emptyset \uplus \fml{0}\).
  Since the neutral element axiom imposes \(\fml{x}_\emptyset\) is summable, there is some \(x' \in X\) such that \(\Sigma \fml{x}_\emptyset \keq x'\) and, since \(\Sigma \fml{0} \keq 0\) and \(\Sigma \{x',0\} \keq x'\) have been established above, we may apply flattening to obtain \(\Sigma \fml{x} \keq x'\).
  But \(\Sigma \fml{x} \keq x\), so it follows that \(x = x'\) and \(\Sigma \fml{x}_\emptyset \keq x\), as claimed.
\end{proof}

\begin{example} \label{ex:wSm_pm}
  Let \(\pm\) be the set \(\{0,+,-\}\) and, for each family \(\fml{x} \in \pm^*\), let \(n_+(\fml{x})\) (and \(n_-(\fml{x})\)) be the number of occurrences of \(+\) (and \(-\)). Define a partial function \(\Sigma \colon \pm^* \pto \pm\) as follows:
  \begin{equation}
    \Sigma \fml{x} = \begin{cases}
      0 &\ifc n_+(\fml{x}) < \infty \text{ and } n_+(\fml{x}) = n_-(\fml{x}) \\
      + &\ifc n_+(\fml{x}) < \infty \text{ and } n_+(\fml{x}) = n_-(\fml{x})+1 \\
      - &\ifc n_+(\fml{x}) < \infty \text{ and } n_+(\fml{x}) = n_-(\fml{x})-1 \\
      \undefined &\otherwise.
    \end{cases}
  \end{equation}
  Then \((\pm,\Sigma)\) is a weak \(\Sigma\)-monoid.
\end{example}

\begin{example} \label{ex:wSm_ominus}
  Let \(X\) be a set and let \(P = \mathcal{P}(X)\) be its power set. For any \(x \in X\), define \(n_x \colon P^* \to \Nset \cup \{\infty\}\) as follows:
  \begin{equation}
    n_x(\{A_i\}_I) = \abs{\{i \in I \mid x \in A_i\}}.
  \end{equation}
  Define a partial function \(\Sigma \colon P^* \pto P\) as follows:
  \begin{equation}
    \Sigma \{A_i\}_I = \begin{cases}
      \{x \in X \mid n_x(\{A_i\}_I) \text{ odd}\} &\ifc \forall x \in X,\, n_x(\{A_i\}_I) < \infty \\
      \undefined &\otherwise.
    \end{cases}
  \end{equation}
  Then \((P,\Sigma)\) is a weak \(\Sigma\)-monoid. On finite families, \(\Sigma\) corresponds to symmetric difference of sets.
\end{example}

In Section~\ref{sec:topology} we establish that every Hausdorff commutative monoid is a weak \(\Sigma\)-monoid.
Many examples of weak \(\Sigma\)-monoids arise like this, and here is one of the simplest.

\begin{example} \label{ex:wSm_Rset}
  Let \((\Rset,+)\) be the standard Hausdorff commutative monoid of real numbers. According to Proposition~\ref{prop:wSm_Hausdorff}, we may define a partial function \(\Sigma \colon \Rset^* \pto \Rset\) so that \((\Rset,\Sigma)\) is a weak \(\Sigma\)-monoid.
  This \(\Sigma\) corresponds to the usual notion of absolute convergence of series.
\end{example}

Let us introduce a notion of structure-preserving function between weak \(\Sigma\)-monoids.

\begin{definition}
  Let \((X,\Sigma)\) and \((Y,\Sigma')\) be weak \(\Sigma\)-monoids. A function \(f \colon X \to Y\) is a \emph{\(\Sigma\)-homomorphism} if for every family \(\fml{x} \in X^*\) and every \(x \in X\):
  \begin{equation*}
    \Sigma \fml{x} \keq x \implies \Sigma' f\fml{x} \keq f(x).
  \end{equation*}
\end{definition}

Notice that in the weak \(\Sigma\)-monoid \((\pm,\Sigma)\) from Example~\ref{ex:wSm_pm} the family \(\{+,+,-\}\) is summable whereas its subfamily \(\{+,+\}\) is not.
Similarly, we may define a weak \(\Sigma\)-monoid \(([-1,1],\Sigma)\) whose \(\Sigma\) function is the same as in \((\Rset,\Sigma)\) when the result is in the interval \([-1,1]\) and it is undefined otherwise.
Once again, \(\{0.75,0.5,-0.25\}\) is summable whereas its subfamily \(\{0.75,0.5\}\) is not.
This idea of defining a weak \(\Sigma\)-monoid from an existing one by restricting the underlying set is formalised by the following lemma.

\begin{lemma} \label{lem:wSm_restriction}
  Let \((Y,\Sigma)\) be a weak \(\Sigma\)-monoid and let \(X\) be a set. Let \(f \colon X \to Y\) be an injective function such that \(0 \in \im(f)\). Define a partial function \(\Sigma^{f} \colon X^* \pto X\) as follows:
  \begin{equation*}
    \Sigma^f \fml{x} = \begin{cases}
      x &\ifc \exists x \in X \colon \Sigma f\fml{x} \keq f(x) \\
      \undefined &\otherwise.
    \end{cases}
  \end{equation*}
  Then, \((X,\Sigma^f)\) is a weak \(\Sigma\)-monoid and \(f \colon X \to Y\) is a \(\Sigma\)-homomorphism.
\end{lemma} \begin{proof}
  The definition of \(\Sigma^f\) is unambiguous thanks to the requirement that \(f\) is injective.
  Neutral element and singleton axioms in \((X,\Sigma^f)\) are trivially derived from those in \((Y,\Sigma)\).
  Let \(\{\fml{x}_j\}_J\) be an indexed set where each \(\fml{x}_j \in X^*\) is a summable family in \((X,\Sigma^f)\) and let \(\fml{x} = \uplus_J \fml{x}_j\), evidently, \(f\fml{x} = \uplus_J f\fml{x}_j\).
  Assume that each family \(\fml{x}_j\) is finite, then the following sequence of implications proves bracketing:
  \begin{align*}
    \Sigma^f \fml{x} \keq x &\iff \Sigma f\fml{x} \keq f(x) &&\text{(definition of \(\Sigma^f\))} \\
        &\implies \Sigma \{\Sigma f\fml{x}_j\}_J \keq f(x) &&\text{(bracketing in \(Y\))} \\
        &\iff \Sigma \{f(\Sigma^f \fml{x}_j)\}_J \keq f(x) &&\text{(\(\fml{x}_j\) summable and def. of \(\Sigma^f\))} \\
        &\iff \Sigma^f \{\Sigma^f \fml{x}_j\}_J \keq x &&\text{(definition of \(\Sigma^f\))}.
  \end{align*}
  Finally, flattening in \((X,\Sigma^f)\) can be proven following a similar argument, this time using flattening in \((Y,\Sigma)\).
  Thus, we have shown that \((X,\Sigma^f)\) is a weak \(\Sigma\)-monoid.
  The fact that \(f \colon X \to Y\) is a \(\Sigma\)-homomorphism is immediate from the definition of \(\Sigma^f\).
\end{proof}

The remainder of this section will study the category of weak \(\Sigma\)-monoids and certain important full subcategories of it.

\begin{definition}
  Let \(\SCat{w}\) be the category whose objects are weak \(\Sigma\)-monoids and whose morphisms are \(\Sigma\)-homomorphisms.
\end{definition}

\subsection{Important full subcategories of \(\SCat{w}\)}
\label{sec:SCat*}

We now introduce subclasses of \(\Sigma\)-monoids whose axioms require more families to be summable.

\begin{definition} \label{def:sSm}
  A \emph{strong \(\Sigma\)-monoid} is a weak \(\Sigma\)-monoid \((X,\Sigma)\) that satisfies the following extra axioms.
  \begin{description}
    \item[{Subsummability.}] Every subfamily \(\fml{x'} \subseteq \fml{x}\) of a summable family \(\fml{x} \in X^*\) is summable.
    \item[{Strong bracketing.}] If \(\{\fml{x}_j\}_J\) is an indexed set where each \(\fml{x}_j \in X^*\) is a summable family, and \(\fml{x} = \uplus_J \fml{x}_j\), then:
    \begin{equation*}
       \Sigma \fml{x} \keq x \implies \Sigma \{\Sigma \fml{x}_j\}_J \keq x
    \end{equation*}
    \item[{Strong flattening.}] If \(\{\fml{x}_j\}_J\) is an indexed set where each \(\fml{x}_j \in X^*\) is a summable family, and \(\fml{x} = \uplus_J \fml{x}_j\), then:
    \begin{equation*}
      \Sigma \{\Sigma \fml{x}_j\}_J \keq x \implies \Sigma \fml{x} \keq x
    \end{equation*}
  \end{description}
\end{definition}

Notice how the finiteness preconditions required in the standard bracketing and flattening axioms are removed, thus making the axioms stronger.
It is straightforward to check that this definition is equivalent to that of Haghverdi~\cite{Haghverdi}.
In Haghverdi's definition, strong bracketing, strong flattening and subsummability are condensed into a single axiom, and the neutral element axiom of weak \(\Sigma\)-monoids is omitted since it follows immediately from subsummability.
Strong $\Sigma$-monoids do not admit additive inverses, as pointed out in the original reference~\cite{Haghverdi}. 
We reproduce this result here, illustrating that it is a consequence of strong flattening.

\begin{proposition}[Haghverdi~\cite{Haghverdi}] \label{prop:no_inverses}
Let \((X,\Sigma)\) be a strong \(\Sigma\)-monoid.
If a family \(\fml{x} \in X^*\) satisfies \(\Sigma \fml{x} \keq 0\), then every element \(a \in \fml{x}\) is the neutral element \(a = 0\).
\end{proposition} \begin{proof}
  Let \(\fml{x}\) be a family such that \(\Sigma \fml{x} \keq 0\) and choose an arbitrary element \(a \in \fml{x}\). Define the following infinite families:
  \begin{equation*}
    \begin{aligned}
      \fml{z} &= \fml{x} \uplus \fml{x} \uplus \dots \\
      \fml{a} &= \{a\} \uplus \fml{z}
    \end{aligned}
  \end{equation*}
  Due to strong flattening, neutral element and singleton axioms:
  \begin{equation*}
    \begin{aligned}
      \Sigma \fml{z} &\keq \Sigma \{\Sigma \fml{x}, \Sigma \fml{x}, \dots \} \keq \Sigma \{0,0,\dots\} \keq 0\\
      \Sigma \fml{a} &\keq \Sigma \{\Sigma \{a\}, \Sigma \fml{z}\} \keq \Sigma \{a,0\} \keq a
    \end{aligned}
  \end{equation*}
  However, both families \(\fml{z}\) and \(\fml{a}\) contain the same elements: a countably infinite number of copies of each element in \(\fml{x}\). Being the same family, it is immediate that \(\Sigma \fml{a} \keq \Sigma \fml{z}\) which implies \(a = 0\).
\end{proof}

Since weak \(\Sigma\)-monoids do admit inverses, it is relevant to consider the notion of \(\Sigma\)-groups. It is immediate that the axioms of \(\Sigma\)-groups defined here are equivalent to those proposed by Higgs~\cite{Higgs}.

\begin{definition}
  A \emph{finitely total \(\Sigma\)-monoid} is a weak \(\Sigma\)-monoid where every finite family is summable.
  A \emph{\(\Sigma\)-group} is a finitely total \(\Sigma\)-monoid \((X,\Sigma)\) where, for every \(x \in X\), there is an element \(-x \in X\) satisfying \(\Sigma \{x,-x\} = 0\) and where the function that maps each \(x \in X\) to \(-x\) is a \(\Sigma\)-homomorphism.
\end{definition}

\begin{remark} \label{rmk:SCat_inverse}
  Notice that every finitely total \(\Sigma\)-monoid \((X,\Sigma)\) is trivially a commutative monoid \((X,+,\Sigma \emptyset)\) where \(x + y\) is defined to be \(\Sigma \{x,y\}\) for every \(x, y \in X\).
  Similarly, every \(\Sigma\)-group is trivially an abelian group and, hence, the inverse of an element in a \(\Sigma\)-group is unique.
  Moreover, thanks to the inversion map being a \(\Sigma\)-homomorphism, any summable family \(\fml{x} \in X^*\) (even infinite ones) have an `inverse' \(\fml{-x}\) obtained by applying the inversion map to each of its elements. Indeed, thanks to flattening:
\begin{equation*}
  \Sigma (\fml{x} \uplus \fml{-x}) \keq \Sigma \{\Sigma \fml{x}, \Sigma \fml{-x}\} \keq \Sigma \{\Sigma \fml{x}, -(\Sigma \fml{x})\} \keq 0.
\end{equation*}
\end{remark}

Each of these flavours of \(\Sigma\)-monoids induces a full subcategory of \(\SCat{w}\).

\begin{definition} \label{def:SCat_subcats}
  Let \(\SCat{s}\), \(\SCat{ft}\) and \(\SCat{g}\) be the full subcategories of \(\SCat{w}\) obtained by restricting their class of objects to strong \(\Sigma\)-monoids, finitely total \(\Sigma\)-monoids, and \(\Sigma\)-groups, respectively.
  For brevity, we refer to these subcategories, along with \(\SCat{w}\) itself, collectively as \(\SCat{*}\).
\end{definition}

Since strong \(\Sigma\)-monoids do not admit inverses (Proposition~\ref{prop:no_inverses}) it is clear that \(\Sigma\)-groups and strong \(\Sigma\)-monoids are disjoint subclasses of weak \(\Sigma\)-monoids.
The hierarchy of inclusions between \(\SCat{*}\) categories can be summarised by the existence of the following full and faithful embedding functors:
\begin{equation} \label{eq:SCat_inclusions}
  \SCat{s} \into \SCat{w} \quad\quad\quad\quad \SCat{g} \into \SCat{ft} \into \SCat{w}.
\end{equation}
We refer to these as the embedding functors between $\SCat{*}$ categories.

\begin{proposition} \label{prop:SCat*_locally_presentable}
  All of the \(\SCat{*}\) categories are locally presentable.
  Consequently, they are all complete and cocomplete.
\end{proposition} \begin{proof}
    Appendix~\ref{sec:locally_presentable} shows that each of the \(\SCat{*}\) categories is equivalent to the category of models of some \(\aleph_1\)-ary essentially algebraic theory, and hence locally \(\aleph_1\)-presentable~\cite[Theorem~3.36]{Adamek_Rosicky}.
    Locally presentable categories are complete~\cite[Corollary 1.28]{Adamek_Rosicky} and cocomplete by definition.
\end{proof}

\begin{proposition} \label{prop:preserve_limits_chain_colimits}
  All of the embedding functors between $\SCat{*}$ categories preserve limits and \(\aleph_1\)-directed colimits.
\end{proposition} \begin{proof}
    The limit of a diagram $D \colon I \to \SCat{w}$ has as its underlying set the limit of $U \circ D$ where $U$ is the forgetful functor $\SCat{w} \to \Set$. However, the forgetful functor $U$ does not create limits, since there are multiple choices of $\Sigma$ functions that make the limit set a cone in $\SCat{w}$.
    Instead, we explicitly provide the $\Sigma$ function for the limit in $\SCat{w}$ and observe that, if all objects in the diagram are in some category $\SCat{*}$, then the limit is also in the same $\SCat{*}$ category. This, along with $\SCat{w}$ being complete, immediately implies that every embedding functor between $\SCat{*}$ categories preserves limits.
    A similar argument follows for directed colimits, where it will be enough to prove the result for colimits of chains, since these are preserved by a functor exactly when directed colimits are preserved~\cite[Corollary 1.7]{Adamek_Rosicky}.
    
    \textit{Equalisers.} Let \(X\) and \(Y\) be weak \(\Sigma\)-monoids and let \(f,g \colon X \to Y\) be \(\Sigma\)-homomorphisms.
    The equaliser set is \(E = \{x \in X \mid f(x) = g(x)\}\), where the inclusion \(e \colon E \to X\) is injective and \(0 \in E\) because \(f(\Sigma \emptyset) = g(\Sigma \emptyset)\).
    Then, according to Lemma~\ref{lem:wSm_restriction}, \(E\) can be endowed with a weak \(\Sigma\)-monoid structure:
    \begin{equation} \label{eq:SCat_equaliser}
    \Sigma^e \fml{x} = \begin{cases}
      x &\ifc \exists x \in E \colon \Sigma e\fml{x} \keq e(x) \\
      \undefined &\text{otherwise}
    \end{cases}
    \end{equation}
    so that \(e \colon E \to X\) is a \(\Sigma\)-homomorphism.
    To prove that \((E,\Sigma^e)\) is an equaliser in \(\SCat{w}\) it only remains to show that, for any other cone $A$, the unique function \(m \colon A \to E\) is a \(\Sigma\)-homomorphism. But this is straightforward since $\Sigma^e$ is, by construction, the partial function with the largest domain of definition for which $e$ is a $\Sigma$-homomorphism. In more detail: whenever a family $\fml{a} \in A^*$ is summable, this property of $\Sigma^e$ will ensure that $m \fml{a}$ is summable in $E$, and the results will match by virtue of the diagram commuting in $\Set$ and every morphism involved being a $\Sigma$-homomorphism.
    Since $E$ inherits its $\Sigma$ function from $X$, it is straightforward to see that $E$ belongs to the same $\SCat{*}$ category as $X$.
    Consequently, the embedding functors between $\SCat{*}$ categories preserve equalisers.

    \textit{Products.} Let \((X,\Sigma^X)\) and \((Y,\Sigma^Y)\) be objects in \(\SCat{w}\).
    Let \(X \times Y\) be the Cartesian product of sets and \(\pi_X\colon X \times Y \to X\) and \(\pi_Y\colon X \times Y \to Y\) its projections. Define a partial function \(\Sigma^\times \colon (X \times Y)^* \pto X \times Y\) as follows:
    \begin{equation} \label{eq:SCat_product}
    \Sigma^\times \fml{p} = \begin{cases}
      (\Sigma^X \pi_X\fml{p}, \Sigma^Y \pi_Y\fml{p}) &\ifc \pi_X\fml{p} \text{ and } \pi_Y\fml{p} \text{ are summable} \\
      \undefined &\otherwise.
    \end{cases}
    \end{equation}
    Then \((X \times Y,\Sigma^\times)\) is a weak \(\Sigma\)-monoid; the proof is straightforward.
    It is immediate that projections are \(\Sigma\)-homomorphisms, so this is a cone.
    To prove that \((X \times Y,\Sigma^\times)\) is a product in \(\SCat{w}\) it only remains to show that the unique function \(m \colon A \to X \times Y\) for any cone $A$ is a \(\Sigma\)-homomorphism, which follows from observing that its components $\pi_X \circ m$ and $\pi_Y \circ m$ must be $\Sigma$-homomorphisms for $A$ to be a cone.
    Notice that if both \((X,\Sigma^X)\) and \((Y,\Sigma^Y)\) belong to a common $\SCat{*}$ category, then so does \((X \times Y,\Sigma^\times)\).
    Consequently, the embedding functors between $\SCat{*}$ categories preserve products, and it is straightforward to generalise the above argument to show that these functors preserve all small products.
    Together with the previous result on equalisers, this implies that the embedding functors between $\SCat{*}$ categories preserve limits.

    \textit{\(\aleph_1\)-directed colimits.} Let 
    \[X_0 \to X_1 \to X_2 \to \ldots\]
    be a \(\aleph_1\)-directed chain of $\Sigma$-homomorphisms in $\SCat{w}$, and denote the morphism $X_i \to X_j$ obtained by composition as $f_{ij}$, with $f_{ii} = \id_{X_i}$. The colimit of such a chain in $\Set$ is the quotient of the disjoint union of all sets in the chain $(\uplus_i X_i) /\!\sim$ where, for $x \in X_i$ and $x' \in X_j$, we let $x \sim x'$ iff there is some $X_k$ in the chain in which $f_{ik}(x) = f_{jk}(x')$. The morphisms $X_i \to (\uplus_j X_j) /\!\sim$ of the colimit in $\Set$ are the functions mapping each $x \in X_i$ to its equivalence class. 

    Let $[\fml{x}]$ be a family of elements of $(\uplus_i X_i) /\!\sim$ and choose representatives for all equivalence classes in the family such that $\fml{x} \in X_j^*$ for some $j$.
    Then, define the partial function $\Sigma$ on $(\uplus_i X_i) /\!\sim$ as follows:
    \begin{equation} \label{eq:SCat_chain}
    \Sigma [\fml{x}] = \begin{cases}
      [x] &\ifc \exists k \; \exists x \in X_k \colon \Sigma^k f_{jk} \fml{x} \keq x \\
      \undefined &\otherwise
    \end{cases}
    \end{equation}
    where $X_k$ is an object in the chain and $\Sigma^k \colon X_k^* \pto X_k$ is its corresponding partial function.
    Notice that as soon as there is a $\Sigma$-monoid $X_k$ where $f_{jk} \fml{x}$ is summable, the family $f_{jk'} \fml{x}$ will be summable for all $k' > k$, thanks to the morphisms in the chain being $\Sigma$-homomorphisms.
    The definition of $\Sigma$ in~\eqref{eq:SCat_chain} is unambiguous since the equivalence relation $\sim$ is defined in terms of equality in the image of $\Sigma$-homomorphisms $f_{ij}$ and, as such, the particular $\Sigma_k$ that is used does not alter the equivalence class we output.
 
    Since the chain is $\aleph_1$-directed and each family $[\fml{x}]$ is countable by definition, there is an upper bound $X_k$ where we can find representatives such that $\fml{x} \in X_k^*$.
    Then, we know the partial function $\Sigma$ defined in~\eqref{eq:SCat_chain} will satisfy the same axioms on $[\fml{x}]$ as $\fml{x}$ does in the $\Sigma$-monoid $X_k$.
    Consequently, if all of the objects in the chain belong to the same $\SCat{*}$ category, then so does $(\uplus_i X_i) /\!\sim$.

    Each of the functions $X_i \to (\uplus_j X_j) /\!\sim$ are trivially $\Sigma$-homomorphisms, so $(\uplus_j X_j) /\!\sim$ is a cocone.
    It only remains to show that for any other cocone $A$, the unique function $m \colon (\uplus_i X_i) /\!\sim\ \to A$ is a $\Sigma$-homomorphism.
    For any summable family $[\fml{x}]$ there is, by definition, an $X_k$ where the corresponding $\fml{x}$ is summable and, since the cocone morphism $X_k \to A$ is a $\Sigma$-homomorphism, it is guaranteed that $m[\fml{x}]$ is summable.

    Thus, $(\uplus_i X_i) /\!\sim$ is a colimit of the \(\aleph_1\)-directed chain and, due to it being constructed in the same way for all of $\SCat{*}$ categories, it is preserved by the embedding functors.
    Together with Corollary 1.7 from~\cite{Adamek_Rosicky}, this implies that the embedding functors between $\SCat{*}$ categories preserve all \(\aleph_1\)-directed colimits.        
\end{proof}

The following theorem establishes that \(\SCat{s}\), \(\SCat{ft}\) and \(\SCat{g}\) are reflective subcategories of \(\SCat{w}\).

\begin{theorem} \label{thm:embedding_right_adjoint}
  Each of the embedding functors between $\SCat{*}$ categories is a right adjoint.
\end{theorem} \begin{proof}
  This follows immediately from~\cite[Theorem 1.66]{Adamek_Rosicky} since we have shown that each of the embedding functors between $\SCat{*}$ categories preserves limits and \(\aleph_1\)-directed colimits, and we know from Proposition~\ref{prop:SCat*_locally_presentable} that each $\SCat{*}$ category is locally presentable.
\end{proof}

Our proof of existence of the left adjoint functors is not constructive.
We provide a construction of the left adjoint to the embedding \(\SCat{s} \into \SCat{w}\) in Appendix~\ref{sec:adjunctions}.
We believe the other left adjoint functors can be constructed using the approach from~\cite[Theorem 29]{partialHorn}, where the authors provide a construction of the left adjoint of any embedding functor between categories of models of quasi-equational theories, provided the objects of the larger category satisfy a subset of the axioms of the smaller category.
A caveat is that the result in~\cite{partialHorn} applies to finitary quasi-equational theories, and in our case we would be dealing with theories of arity $\aleph_1$, so a generalisation of their theorem would be required.

\section{Tensor products of $\Sigma$-monoids}
\label{sec:tensor}

The results in this section are adapted from those in the appendix of Hoshino's work on the category of strong \(\Sigma\)-monoids~\cite{RTUDC}.
These are generalised so that they apply to \(\SCat{w}\), proving that it has a well-defined tensor product that makes it a symmetric monoidal closed category. Similarly to Hoshino~\cite{RTUDC}, we do not provide an explicit construction of the tensor product, but rather restrict to the proof of its existence.
We then prove that every $\SCat{*}$ category has a tensor product obtained by applying the left adjoint functor $F \colon \SCat{*} \to \SCat{w}$ to the tensor product in $\SCat{w}$. Indeed, every $\SCat{*}$ category is symmetric monoidal closed, and the functor $F$ is strong symmetric monoidal.
Some of the proofs are deferred to Appendix~\ref{sec:tensor_appendix}.

Proposition~\ref{prop:SCat*_locally_presentable} already established that \(\SCat{w}\) has small products. Consequently, \(\SCat{w}\) can be given a Cartesian monoidal structure. We will lift this Cartesian product to a tensor product through a notion of \(\Sigma\)-bilinear function.

\begin{definition} \label{def:bihom}
  Let \(X,Y,Z \in \SCat{*}\) be objects. A function \(f \colon X \times Y \to Z\) is \emph{\(\Sigma\)-bilinear} function if \(f(x,-)\) and \(f(-,y)\) are \(\Sigma\)-homomorphisms for all \(x \in X\) and all \(y \in Y\).
  Let \(\SCat{*}^{X,Y}(Z)\) be the set of all \(\Sigma\)-bilinear functions of type \(X \times Y \to Z\).
\end{definition}

For each category \(\SCat{*}\) and every pair of objects \(X,Y \in \SCat{*}\) there is a functor \(\SCat{*}^{X,Y} \colon \SCat{*} \to \Set\) that maps each object \(Z \in \SCat{*}\) to the set of \(\Sigma\)-bilinear functions \(\SCat{*}^{X,Y}(Z)\) and sends each \(\Sigma\)-homomorphism \(f \colon Z \to W\) to a function that maps each \(h \in \SCat{*}^{X,Y}(Z)\) to \(f \circ h \in \SCat{*}^{X,Y}(W)\).

\begin{lemma} \label{lem:SCatw_tensor_main}
  The comma category \(\comma{\{\bullet\}}{\SCat{w}^{X,Y}}\) has an initial object --- we denote it by \((p,X \otimes Y)\).
\end{lemma}
\begin{proof}
  See Lemma~\ref{lem:SCatw_tensor}.
\end{proof}

The previous lemma establishes the existence of a \(\Sigma\)-monoid \(X \otimes Y\) --- known as the tensor product --- and a \(\Sigma\)-bilinear function \(p \colon X \times Y \to X \otimes Y\) that let us uniquely characterise any \(\Sigma\)-bilinear function \(f \colon X \times Y \to Z\) as a \(\Sigma\)-homomorphism \(\bar{f} \colon X \otimes Y \to Z\) so that \(f = \bar{f} \circ p\). We now show that \(\SCat{w}\) has a symmetric monoidal closed structure given by this tensor product.

The monoidal unit of $\SCat{w}$ is $I = (\{0, 1\}, \Sigma)$ with its $\Sigma$ function defined as follows:
\begin{equation}
    \Sigma \{n_j\}_J = \begin{cases}
      0 \quad &\ifc \neg\exists j \in J \colon n_j = 1 \\
      1 \quad &\ifc \exists! j \in J \colon n_j = 1 \\
      \undefined &\otherwise.
    \end{cases}
\end{equation}
For every \(X \in \SCat{w}\) we define functions \(l_X \colon I \times X \to X\) and \(r_X \colon X \times I \to X\) where:
\begin{equation} \label{eq:lr_bihoms}
  \begin{aligned}
    l_X(0,x) &= 0 = r_X(x,0) \\
    l_X(1,x) &= x = r_X(x,1) \\
  \end{aligned}
\end{equation}
It is straightforward to check that \(l_X\) and \(r_X\) are both \(\Sigma\)-bilinear functions.

\begin{corollary} \label{cor:tensor_unitors}
  Let \(l_X\) and \(r_X\) be the \(\Sigma\)-bilinear functions defined in~\eqref{eq:lr_bihoms} and let \(a\) be the associator from the Cartesian monoidal structure on \(\SCat{w}\). The following diagram commutes in $\Set$:
  \[\begin{tikzcd}[column sep=tiny]
    {(X \times I) \times Y} && {X \times (I \times Y)} \\
    {X \times Y} && {X \times Y} \\
    & {X \otimes Y}
    \arrow["p", from=2-1, to=3-2]
    \arrow["p"', from=2-3, to=3-2]
    \arrow["{r \times \id}", from=1-1, to=2-1]
    \arrow["{\id \times l}"', from=1-3, to=2-3]
    \arrow["a", from=1-1, to=1-3]
  \end{tikzcd}\]
\end{corollary} \begin{proof}
    Immediate from evaluating the action of both paths on inputs $((x,0),y)$ and $((x,1),y)$.
\end{proof}

These \(l\) and \(r\) will be lifted through \(p\) to define the unitors of the monoidal structure given by \(\otimes\). The associator is trickier since the \(\Sigma\)-homomorphism characterisation of a `\(\Sigma\)-trilinear' function is not yet clear. Here, we say \(f\) is a \emph{\(\Sigma\)-trilinear} function if \(f(x,y,-)\), \(f(x,-,z)\) and \(f(-,y,z)\) are all \(\Sigma\)-homomorphisms for every choice of \(x \in X\), \(y \in Y\) and \(z \in Z\).
In the following, we establish that the object \((X \otimes Y) \otimes Z\) along with the \(\Sigma\)-trilinear function \(p \circ (p \times \id)\) can take up the role of the ternary tensor product.

\begin{lemma} \label{lem:SCat_3_tensor_main}
  For objects \(X,Y,Z,W\) in $\SCat{w}$, let \(f \colon (X \times Y) \times Z \to W\) be a \(\Sigma\)-trilinear function.
  There is a unique \(\Sigma\)-homomorphism \((X \otimes Y) \otimes Z \to W\) making the following diagram commute in $\Set$:
  \[\begin{tikzcd}
    {(X \times Y) \times Z} & W \\
    & {(X \otimes Y) \otimes Z}
    \arrow[dashed, from=2-2, to=1-2]
    \arrow["f", from=1-1, to=1-2]
    \arrow["{p\circ(p \times \id)}"', from=1-1, to=2-2]
  \end{tikzcd}\]
\end{lemma}
\begin{proof}
  See Lemma~\ref{lem:SCat_3_tensor}.
\end{proof}

Similarly, any \(\Sigma\)-trilinear function \(f \colon X \times (Y \times Z) \to W\) factors uniquely through \(p \circ (\id \times p)\), \ie{} both \((X \otimes Y) \otimes Z\) and \(X \otimes (Y \otimes Z)\) are ternary tensor products.
It is finally time to define the monoidal structure on the tensor product.

\begin{definition} \label{def:tensor_lifting}
  Let \(X \xto{f} Z\) and \(Y \xto{g} W\) be morphisms in \(\SCat{w}\). Define \(f \otimes g\), \(\lambda\), \(\rho\) and \(\alpha\) to be the unique \(\Sigma\)-homomorphisms making the following diagrams commute in $\Set$:
  \[\begin{tikzcd}[column sep=small]
    {X \times Y} & {Z \times W} && {I \times X} & X && {X \times I} & X \\
    {X \otimes Y} & {Z \otimes W} && {I \otimes X} &&& {X \otimes I}
    \arrow["{f \times g}", from=1-1, to=1-2]
    \arrow["p"', from=1-1, to=2-1]
    \arrow["p", from=1-2, to=2-2]
    \arrow["{f \otimes g}"', dashed, from=2-1, to=2-2]
    \arrow["l", from=1-4, to=1-5]
    \arrow["p"', from=1-4, to=2-4]
    \arrow["\lambda"', dashed, from=2-4, to=1-5]
    \arrow["\rho"', dashed, from=2-7, to=1-8]
    \arrow["r", from=1-7, to=1-8]
    \arrow["p"', from=1-7, to=2-7]
  \end{tikzcd}\]
  \[\begin{tikzcd}[column sep=small]
    {(X \times Y) \times Z} && {X \times(Y \times Z)} \\
    {(X \otimes Y) \otimes Z} && {X \otimes (Y \otimes Z).}
    \arrow["a", from=1-1, to=1-3]
    \arrow["{p\circ(\id \times p)}", from=1-3, to=2-3]
    \arrow["{p \circ (p \times \id)}"', from=1-1, to=2-1]
    \arrow["\alpha"', dashed, from=2-1, to=2-3]
  \end{tikzcd}\]
\end{definition}

\begin{lemma} \label{lem:SCatw_tensor_monoidal}
  The category \((\SCat{w},\otimes,I)\) is monoidal.
\end{lemma}
\begin{proof}
  On objects \(X,Y \in \SCat{w}\), the functor \(\otimes\) yields \(X \otimes Y\) from Lemma~\ref{lem:SCatw_tensor_main} and, on morphisms \(f\) and \(g\), it yields the morphism \(f \otimes g\) from Definition~\ref{def:tensor_lifting}; it is straightforward to check that this is indeed a functor, thanks to functoriality of the Cartesian product and the universal property of \(X \otimes Y\).
  Unitors \(\lambda\) and \(\rho\) and associators \(\alpha\) are given in Definition~\ref{def:tensor_lifting}.
  The triangle equation follows from Corollary~\ref{cor:tensor_unitors} after lifting all occurrences of \(\times\) to \(\otimes\) using \(p\), whereas the pentagon equation follows from the pentagon equation of the Cartesian monoidal structure on \(\SCat{w}\) along with naturality of \(a\).
  Naturality of \(\lambda\), \(\rho\) and \(\alpha\) follow trivially.
  The inverse of \(\alpha\) is induced from the inverse of \(a\) and the universal property of \(X \otimes (Y \otimes Z)\).
  The inverse of \(\lambda\) is somewhat trickier: let \(i \colon X \to I \times X\) be the function that maps each \(x \in X\) to \((1,x)\); it is immediate that \(l \circ i = \id\), and \(\lambda \circ (p \circ i) = \id\) follows from the definition of \(\lambda\). Notice that \(p \circ i\) is a \(\Sigma\)-homomorphism by virtue of \(p\) being a \(\Sigma\)-bilinear function. To prove that \(p \circ i\) is the inverse of \(\lambda\) it remains to show that \((p \circ i) \circ \lambda = \id\).
  Notice that:  
  \begin{equation*}
    p(1,l(0,x)) = p(1, \Sigma \emptyset) = \Sigma \emptyset = p(0,x)
  \end{equation*}
  \begin{equation*}
    p(1,l(1,x)) = p(1, x)
  \end{equation*}
  so, \(p \circ i \circ l = p\) and, by definition of \(\lambda\), we have \(p \circ i \circ \lambda \circ p = p\); this means that the \(\Sigma\)-bilinear function \(p\) may be characterised by the \(\Sigma\)-homomorphism \(\bar{p} = p \circ i \circ \lambda\). However, \(p\) is also trivially characterised by \(\bar{p} = \id\) and, due to uniqueness of \(\bar{p}\), we conclude that \((p \circ i) \circ \lambda = \id\).
  Therefore, the \(\Sigma\)-homomorphism \(p \circ i\) is the inverse of \(\lambda\) and we may define the inverse of \(\rho\) by similar means.
  This completes the proof.
\end{proof}

To see that $\SCat{w}$ is monoidal closed, we need to introduce a \(\Sigma\)-monoid structure on sets of \(\Sigma\)-homomorphisms.

\begin{lemma} \label{lem:SCat_internal_hom_main}
  For objects \(X,Y \in \SCat{w}\), the hom-set \([X,Y]\) may be endowed with a partial function \(\Sigma^\to \colon [X,Y]^* \pto [X,Y]\) defined as:
  \begin{equation*}
    \Sigma^\to \fml{f} = \begin{cases}
      s_\fml{f} &\ifc s_\fml{f} \in [X,Y] \\
      \undefined &\otherwise
    \end{cases}
  \end{equation*}
  for every family \(\fml{f} = \{f_i\}_I \in [X,Y]^*\), where \(s_\fml{f}(x) = \Sigma \{f_i(x)\}_I\) for each \(x \in X\).
  Then, \(([X,Y], \Sigma^\to) \in \SCat{w}\).
\end{lemma}
\begin{proof}
  See Lemma~\ref{lem:SCat_internal_hom}.
\end{proof}

\begin{proposition} \label{prop:SCatw_SMCC}
  The category \((\SCat{w},\otimes,I)\) is symmetric monoidal closed.
\end{proposition}
\begin{proof}
  The symmetry of the monoidal structure is derived from that of the Cartesian monoidal structure of \(\SCat{w}\), lifted using \(p\).
  Similarly, the closure is obtained by lifting the Cartesian closure of \(\Set\) through the universal property of the tensor product in \(\SCat{w}\).
  This is possible because the evaluation map \(\mathrm{ev} \colon [Y,Z] \times Y \to Z\) in \(\Set\) is a \(\Sigma\)-bilinear function when \([Y,Z], Z \in \SCat{w}\).
  To check this, let \(h \in [Z,W]\) be an arbitrary \(\Sigma\)-homomorphism and \(\fml{z} = \{z_i\}_I \in Z^*\) an arbitrary summable family, then:
  \begin{equation*}
    \mathrm{ev}(h,\Sigma \fml{z}) = h(\Sigma \fml{z}) \keq \Sigma h\fml{z} = \Sigma \{\mathrm{ev}(h,z_i)\}_I.
  \end{equation*}
  Similarly, let \(\fml{h} = \{h_i\}_I \in [Z,W]^*\) be an arbitrary summable family and fix an arbitrary \(z \in Z\), then:
  \begin{equation*}
    \mathrm{ev}(\Sigma^\to \fml{h},z) = \mathrm{ev}(s_\fml{h},z) = s_\fml{h}(z) = \Sigma \{h_i(z)\}_I = \Sigma \{\mathrm{ev}(h_i,z)\}_I.
  \end{equation*}
  Consequently, we can define \(\overline{\mathrm{ev}}\) as the unique \(\Sigma\)-homomorphism such that \(\mathrm{ev} = \overline{\mathrm{ev}} \circ p\).
  Furthermore, \(\overline{\mathrm{ev}}\) is a natural transformation due to naturality of \(\mathrm{ev}\), lifted through \(p\).
  For arbitrary objects \(X, Y, Z \in \SCat{w}\) and any \(\Sigma\)-homomorphism \(f \colon X \otimes Y \to Z\), the following diagram commutes in \(\Set\):
  \[\begin{tikzcd}
    {X \times Y} & {X \otimes Y} && Z \\
    & {[Y,Z] \otimes Y} \\
    {[Y,Z] \times Y}
    \arrow["p", from=1-1, to=1-2]
    \arrow["p", from=3-1, to=2-2]
    \arrow["f", from=1-2, to=1-4]
    \arrow["{\Lambda f' \otimes \id}"', dashed, from=1-2, to=2-2]
    \arrow["{\Lambda f' \times \id}"', dashed, from=1-1, to=3-1]
    \arrow["{\overline{\mathrm{ev}}}"{pos=0.4}, from=2-2, to=1-4]
    \arrow["{\mathrm{ev}}"', curve={height=70pt}, from=3-1, to=1-4]
  \end{tikzcd}\]
  where \(\Lambda f'(x) = f(p(x,-))\) for all \(x \in X\). The outer triangle commutes due to the Cartesian closed structure of \(\Set\), the square and bottom triangle commute by definition of the functor \(\otimes\) and the \(\Sigma\)-homomorphism \(\overline{\mathrm{ev}}\); hence, the upper right triangle commutes as well. Moreover, \(\Lambda f'\) is a \(\Sigma\)-homomorphism due to \(p\) being \(\Sigma\)-bilinear, so the upper right triangle is a commuting diagram in \(\SCat{w}\).
  Uniqueness of \(\Lambda f'\) is imposed by the monoidal closed structure of \(\Set\) and, since every \(\Sigma\)-homomorphism is a function, this implies uniqueness in \(\SCat{w}\), completing the proof.
\end{proof}

Now that the symmetric monoidal closed structure of $\SCat{w}$ has been established, we can show that each of the $\SCat{*}$ subcategories are also symmetric monoidal closed.
Choose a $\SCat{*}$ category and let $U \colon \SCat{*} \into \SCat{w}$ be the corresponding embedding functor. Thanks to Theorem~\ref{thm:embedding_right_adjoint} we know that $U$ has a left adjoint functor $F \colon \SCat{w} \to \SCat{*}$.

\begin{lemma} \label{lem:F(X,Y)}
    The pair $(\eta \circ p, F(X \otimes Y))$ is an initial object of the comma category \(\comma{\{\bullet\}}{\SCat{w}^{X,Y} \!\circ U}\).
\end{lemma}
\begin{proof}
    Let $X,Y \in \SCat{w}$ and $Z \in \SCat{*}$ and consider an arbitrary $\Sigma$-bilinear function $f \colon X \times Y \to U(Z)$.
    The following diagram commutes in $\Set$:
    \[\begin{tikzcd}
        {UF(X \otimes Y)} && {U(Z)} \\
        {X \otimes Y} \\
        {X \times Y}
        \arrow[dashed, from=1-1, to=1-3]
        \arrow["\eta", from=2-1, to=1-1]
        \arrow["{\bar{f}}"{pos=0.4}, dashed, from=2-1, to=1-3]
        \arrow["f"', from=3-1, to=1-3]
        \arrow["p", from=3-1, to=2-1]
    \end{tikzcd}\]
    The lower triangle commutes due to the universal property of the tensor product in $\SCat{w}$ and the upper triangle commutes due to $F \dashv U$.
    Recall that $\bar{f}$ is the unique $\Sigma$-homomorphism making the diagram commute and, by adjunction, the $\Sigma$-homomorphism $F(X \otimes Y) \to Z$ is also uniquely determined.
    Since this argument applies to any $\Sigma$-bilinear function $f \colon X \times Y \to U(Z)$, the proof is complete.
\end{proof}

In particular, for $X, Y \in \SCat{*}$ we have that $(\eta \circ p,\, F(UX \otimes UY))$ is an initial object of the comma category \(\comma{\{\bullet\}}{\SCat{w}^{UX,UY} \circ U}\) or, equivalently $\SCat{*}(F(UX \otimes UY), Z) \iso \SCat{*}^{X,Y}(Z)$. This implies that $F(UX \otimes UY)$ is the tensor product of $\Sigma$-monoids in the category $\SCat{*}$, and it suggests that the monoidal structure in $\SCat{*}$ can be constructed from that of $\SCat{w}$. Indeed, the following results formalise this.

\begin{lemma} \label{lem:strong_F}
For any $\SCat{*}$ subcategory of $\SCat{w}$ and objects $X, Y \in \SCat{*}$, define the object $X \otimes' Y$ as $F(UX \otimes UY)$. There is an isomorphism in $\SCat{*}$ between the following objects:
\begin{equation*}
    F(X \otimes Y) \iso F(X) \otimes' F(Y).
\end{equation*}
\end{lemma} \begin{proof}
    Thanks to Lemma~\ref{lem:F(X,Y)}, we know that for any $Z \in \SCat{*}$ there are isomorphisms
    \begin{align*}
        \SCat{*}(F(X \otimes Y), Z) &\iso \SCat{w}^{X,Y}(U(Z))  \\
        \SCat{*}(F(X) \otimes' F(Y), Z) &\iso \SCat{w}^{UF(X),UF(Y)}(U(Z))
    \end{align*}
    where, for the second one, we have also used the definition of $\otimes'$.
    Consider the following chain of isomorphisms:
    \begin{align*}
        \SCat{w}^{X,Y}(U(Z)) &\iso \SCat{w}(X \otimes Y,\, U(Z)) &&\text{(def. $\otimes$)} \\
          &\iso \SCat{w}(X,\, [Y, U(Z)]) &&\text{(\SCat{w} mon. closed)} \\
          &\iso \SCat{w}(X,\, [UF(Y), U(Z)]) &&\text{($F \dashv U$ reflection*)} \\
          &\iso \SCat{w}(UF(Y),\, [X, U(Z)]) &&\text{(\SCat{w} sym. mon. closed)} \\
          &\iso \SCat{w}(UF(Y),\, [UF(X), U(Z)]) &&\text{($F \dashv U$ reflection*)} \\
          &\iso \SCat{w}(UF(X) \otimes UF(Y),\, U(Z)) &&\text{(\SCat{w} sym. mon. closed)} \\
          &\iso \SCat{w}^{UF(X),UF(Y)}(U(Z)) &&\text{(def. $\otimes$)}
    \end{align*}
    where the only non-trivial steps (*) are those where we use that $F \dashv U$ is a reflection.
    In more detail: the adjunction imposes $\SCat{w}(Y, U(Z)) \iso \SCat{*}(F(Y), Z)$ and since $U$ is full and faithful, there is an isomorphism $\SCat{*}(F(Y), Z) \iso \SCat{w}(UF(Y), U(Z))$.
    Thus, we have a bijective function between the underlying sets of $[Y, U(Z)], [UF(Y), U(Z)] \in \SCat{w}$ and, since the definition of $\Sigma^{\to}$ in Lemma~\ref{lem:SCat_internal_hom_main} only depends on the $\Sigma$-monoid in the codomain --- which is $U(Z)$ for both of these --- we conclude that $[Y, U(Z)] \iso [UF(Y), U(Z)]$ in $\SCat{w}$, which is what we require in the steps above. 
    
    Notice that all of these isomorphisms are natural isomorphisms on functors $\SCat{*} \to \Set$, where so far we have been applying those functors to $Z \in \SCat{*}$.
    Thus, we have shown the existence of a natural isomorphism $\SCat{*}(F(X \otimes Y), -) \iso \SCat{*}(F(X) \otimes' F(Y), -)$ and it follows from the Yoneda embedding that $F(X \otimes Y) \iso F(X) \otimes' F(Y)$, completing the proof.
\end{proof}

\begin{theorem} \label{thm:SCat*_SMCC}
    Each $\SCat{*}$ category is symmetric monoidal closed using the tensor product as its monoidal product.
    The left adjoint functors of the embeddings between $\SCat{*}$ categories are strong symmetric monoidal.
\end{theorem} \begin{proof}
    The claim for $\SCat{w}$ has already been proven in Proposition~\ref{prop:SCatw_SMCC}.
    For each $\SCat{*}$ subcategory of $\SCat{w}$, use $X \otimes' Y$ from Lemma~\ref{lem:strong_F} as the monoidal product and $F(I)$ as the monoidal unit. For any $X \in \SCat{w}$, the component of the left unitor $\lambda'_{F(X)} \colon F(I) \otimes' F(X)$ is defined as $F(\lambda_X)$ conjugated with the isomorphism from Lemma~\ref{lem:strong_F}; for any other object $Y \in \SCat{*}$ the component $\lambda'_Y$ is defined in terms of $\lambda'_{FU(Y)}$, using that the counit is an isomorphism. The right unitor, associator and symmetric braiding are defined similarly. This makes $(\SCat{*}, \otimes', F(I))$ a symmetric monoidal category, with commutation of its coherence diagrams being a direct consequence of commutation of the same diagrams in $(\SCat{w}, \otimes, I)$. The functor $F$ is strong symmetric monoidal by definition.
    It only remains to prove that $(\SCat{*}, \otimes', F(I))$ is monoidal closed, but this follows from the fact that $F$ is strong symmetric monoidal, due to a well-known consequence of Day's reflection theorem~\cite{Day}.
\end{proof}

\section{Hausdorff commutative monoids are \(\Sigma\)-monoids}
\label{sec:topology}

As promised in Section~\ref{sec:SCat}, every Hausdorff commutative monoid \((X,\tau,+)\) can be given a partial function \(\Sigma \colon X^* \pto X\) such that \((X,\Sigma)\) is a finitely total \(\Sigma\)-monoid whose summable families are precisely those that converge according to the topology $\tau$.

\begin{definition} \label{def:topological_Sigma}
  Let \((X,\tau,+)\) be a Hausdorff commutative monoid. For any family \(\fml{x} \in X^*\), let \((\fpset{\fml{x}},\subseteq)\) be the directed set of finite subfamilies of \(\fml{x}\) and let \(\sigma \colon \fpset{\fml{x}} \to X\) be the net that maps each finite subfamily \(\fml{x'} \in \fpset{\fml{x}}\) to the sum of its elements, obtained after a finite number of applications of \(+\).
  The \emph{extended monoid operation} \(\Sigma \colon X^* \pto X\) is defined as follows:
  \begin{equation*}
    \Sigma \fml{x} = \begin{cases}
      x \quad &\ifc \exists x \in X \colon \lim \sigma = x \\
      \undefined &\otherwise.
    \end{cases}
  \end{equation*}
\end{definition}

We now show that the extended monoid operation satisfies the axioms of finitely total \(\Sigma\)-monoids.

\begin{theorem} \label{prop:wSm_Hausdorff}
  If \((X,\tau,+)\) is a Hausdorff commutative monoid, and \(\Sigma\) its extended monoid operation, then \((X,\Sigma)\) is a finitely total \(\Sigma\)-monoid.
\end{theorem} \begin{proof}
  It is immediate from the definition of the extended monoid operation that every finite family \(\fml{x} \in X^*\) is summable: \(\Sigma \fml{x} \keq \sigma(\fml{x})\).
  In particular, the singleton axiom of weak \(\Sigma\)-monoids follows immediately from \(\Sigma \{x\} \keq x\), and furthermore \(\Sigma \emptyset \keq 0\) where \(0\) is the neutral element of the monoid \(X\).
  Hence the neutral element axiom of weak \(\Sigma\)-monoids is also straightforward to check: the number of occurrences of \(0\) in a family \(\fml{x}\) does not affect the limiting behaviour of \(\sigma\).

  For the flattening axiom of weak \(\Sigma\)-monoids, first consider the simpler case \(\abs{J} = 2\): let \(\fml{x}_0, \fml{x}_1 \in X^*\) be two summable families and assume \(\Sigma \{\Sigma \fml{x}_0, \Sigma \fml{x}_1\} \keq x\) for some \(x \in X\). That \(\fml{x}_0\) and \(\fml{x}_1\) are summable implies that their corresponding nets \(\sigma_0 \colon \fpset{\fml{x}_0} \to X\) and \(\sigma_1 \colon \fpset{\fml{x}_1} \to X\) have a limit point. It follows from the assumption that \(\lim \sigma_0 + \lim \sigma_1 = x\) but, unfortunately, we cannot immediately use that \(+\) is continuous due to \(\sigma_0\) and \(\sigma_1\) being defined on different domains. However, we can define a subnet \(\hat{\sigma}_0 \colon \fpset{\fml{x}_0 \uplus \fml{x}_1} \to X\) as \(\hat{\sigma}_0 = \sigma_0 \circ m_0\) where \(m_0(\fml{x'}) = \fml{x'} \cap \fml{x}_0\); it follows from Proposition~\ref{prop:subnet_limit} that
  \begin{equation*}
    \lim \hat{\sigma}_0 + \lim \hat{\sigma}_1 \keq x.
  \end{equation*}
  Since both \(\hat{\sigma}_0\) and \(\hat{\sigma}_1\) have the same domain, continuity of \(+\) implies
  \begin{equation*}
    \lim\,(\hat{\sigma}_0 + \hat{\sigma}_1) \keq x
  \end{equation*}
  and, due to associativity and commutativity of \(+\), the net \(\hat{\sigma}_0 + \hat{\sigma}_1\) is the same as the net \(\sigma\) on the family \(\fml{x}_0 \uplus \fml{x}_1\), implying \(\Sigma (\fml{x}_0 \uplus \fml{x}_1) \keq x\).
  In the general case, the flattening axiom considers a finite set of summable families \(\{\fml{x}_j\}_J\). If \(\abs{J} > 2\) we may repeat the argument above as many times as necessary to aggregate all subfamilies.

  Finally, the bracketing axiom of weak \(\Sigma\)-monoids requires that, given a family \(\fml{x} \in X^*\) partitioned into a set of finite families \(\fml{x} = \uplus_J \fml{x}_j\) such that \(\Sigma \fml{x}_j \keq x_j\), we show that \(\Sigma \fml{x} \keq x\) implies \(\Sigma \{x_j\}_J \keq x\).
  Let \(\hat{\sigma} \colon \fpset{J} \to X\) be the net that maps each finite subset \(J' \subseteq J\) to the sum of the finite family \(\{x_j\}_{J'}\).
  Assume that \(\Sigma \fml{x} \keq x\) and let \(\sigma \colon \fpset{\fml{x}} \to X\) be its corresponding net, satisfying \(\lim \sigma \keq x\).
  Notice that \(\Sigma \{x_j\}_J \keq \lim \hat{\sigma}\) so our goal is to show that \(\lim \sigma = x\) implies \(\lim \hat{\sigma} = x\).
  Define a function \(m \colon \fpset{J} \to \fpset{\fml{x}}\) as follows:
  \begin{equation*}
    m(J') = \uplus_{J'} \fml{x}_j.
  \end{equation*}
  It is straightforward to check that \(\hat{\sigma} = \sigma \circ m\), making \(\hat{\sigma}\) a subnet of \(\sigma\).
  It then follows from Proposition~\ref{prop:subnet_limit} that \(\lim \hat{\sigma} = x\).

  Thus, we have shown that \((X,\Sigma)\) is a weak \(\Sigma\)-monoid and, in particular, a finitely total \(\Sigma\)-monoid since \(\Sigma \fml{x} \keq \sigma(\fml{x})\) for every finite family \(\fml{x} \in X^*\).
\end{proof}

Continuous monoid homomorphisms between Hausdorff commutative monoids become \(\Sigma\)-homomorphisms between the \(\Sigma\)-monoids induced by their extended monoid operation.

\begin{proposition}
  If \(X\) and \(Y\) are Hausdorff commutative monoids, and \(f \colon X \to Y\) is a continuous monoid homomorphism, then for every \(\fml{x} \in X^*\) and \(x \in X\):
  \begin{equation*}
    \Sigma \fml{x} \keq x \implies \Sigma f\fml{x} \keq f(x)
  \end{equation*}
\end{proposition}
\begin{proof}
  Let \(\sigma \colon \fpset{\fml{x}} \to X\) be the net of finite partial sums of \(\fml{x}\) and assume that \(\Sigma \fml{x} \keq x\).
  Then \(\lim \sigma = x\) and, due to \(f\) being continuous, \(\lim\, (f \circ \sigma) = f(x)\).
  Considering that \(f\) is a monoid homomorphism, it is straightforward to check that \(f \circ \sigma\) is the same function as the corresponding net of \(f\fml{x}\).
  It then follows that \(\Sigma f\fml{x} \keq f(x)\), proving the claim.
\end{proof}

\begin{definition}
  Write \(\HausCMon\) for the category whose objects are Hausdorff commutative monoids and whose morphisms are continuous monoid homomorphisms, and \(\HAG\) for the full subcategory whose objects are Hausdorff abelian groups.
\end{definition}

\begin{corollary} \label{cor:HausCMon_ft}
  There is a faithful functor \(G \colon \HausCMon \to \SCat{ft}\) that equips each Hausdorff commutative monoid with the extended monoid operation from Definition~\ref{def:topological_Sigma} and acts as the identity on morphisms.
\end{corollary}

\begin{proposition}
  The functor \(G \colon \HausCMon \to \SCat{ft}\) from Corollary~\ref{cor:HausCMon_ft} is not full.
\end{proposition} \begin{proof}
  Let \(A \in \HausCMon\) have underlying set \(\{0\} \cup [1,\infty)\) along with the discrete topology and standard addition of the real numbers.
  Let \(B \in \HausCMon\) have the same underlying set \(\{0\} \cup [1,\infty)\) along with standard addition, but whose topology is given by the base of open sets comprised of the singleton \(\{0\}\) and every interval \([x, x+\epsilon)\) for each \(x \in [1,\infty)\) and each \(\epsilon > 0\).

  Notice that \(\Sigma\) on both \(G(A)\) and \(G(B)\) is only defined on families with a finite number of nonzero elements. Since both \(G(A)\) and \(G(B)\) have the same summable families, it is immediate that \(G(A) = G(B)\). Consider the identity function on the underlying sets \(\id \colon B \to A\); it is immediate this is a valid morphism \(G(B) \to G(A)\) in \(\SCat{ft}\), and, since \(G\) acts as the identity on morphisms, for \(G\) to be full we would need \(\id \colon B \to A\) to be continuous. However, \(\id \colon B \to A\) is not continuous due to the topology of \(B\) being strictly coarser than that of \(A\).
  Therefore, \(G \colon \HausCMon \to \SCat{ft}\) is not full.
\end{proof}

Similarly, there is a faithful functor \(\HAG \to \SCat{g}\) since its continuity implies that inversion is a \(\Sigma\)-homomorphism. We expect this functor not to be full either.
The following subsection establishes that both functors have a left adjoint.

\subsection{Free Hausdorff commutative monoids on $\Sigma$-monoids}
\label{sec:HausCMon->SCatft}

The functor \(G \colon \HausCMon \to \SCat{ft}\) of Corollary~\ref{cor:HausCMon_ft} has a left adjoint.
An explicit construction has not been achieved --- instead, our proof uses the general adjoint functor theorem~\cite[Theorem~6.3.10]{Leinster}. We need the following preliminary result.

\begin{proposition} \label{prop:HausCMon_complete}
  The category \(\HausCMon\) is complete, and the functor \\ \(G \colon \HausCMon \to \SCat{ft}\) from Corollary~\ref{cor:HausCMon_ft} preserves limits.
\end{proposition}
\begin{proof}
  See Appendix~\ref{sec:HausCMon->SCatft_appendix}.
\end{proof}

\begin{theorem} \label{thm:HCMon_ft_adj}
  There is a left adjoint to the functor \(G \colon \HausCMon \to \SCat{ft}\).
\end{theorem} \begin{proof}
  Clearly \(\HausCMon\) is locally small since there is a faithful forgetful functor \(\HausCMon \to \Set\).
  By Proposition~\ref{prop:HausCMon_complete}, \(\HausCMon\) is complete and \(G\) preserves limits.
  It remains to show that, for every \(X \in \SCat{ft}\), the comma category \(\comma{X}{G}\) has a weakly initial set - the existence of a left adjoint then follows from the general adjoint functor theorem~\cite[Theorem~6.3.10]{Leinster}.
  Let \(S\) be the collection of all \(\Sigma\)-homomorphisms \(X \to G(A)\) where \(A\) may be any Hausdorff commutative monoid on a set quotient of \(X\).
  Notice that \(S\) is small since all quantifiers in its definition are with respect to a fixed set \(X\). We prove below that \(S\) is a weakly initial set of the comma category \(\comma{X}{G}\).

  Let \(Y\) be an arbitrary Hausdorff commutative monoid and let \(f \colon X \to G(Y)\) be a \(\Sigma\)-homomorphism.
  Recall that every function can be factored through its image as follows:
  \begin{equation*}
    f = X \xto{q} X/{\sim} \xto{\bar{f}} \im(f) \xto{u} G(Y)
  \end{equation*}
  where \(x \sim x'\) iff \(f(x) = f(x')\). Here, \(q\) is surjective, \(\bar{f}\) is bijective and \(u\) is injective.
  First, we provide the set \(X/{\sim}\) with the structure of a Hausdorff commutative monoid.
  Recall that \(G(Y)\) has the same underlying set as \(Y\), so \(\im(f) \subseteq Y\) and, thanks to \(u\) and \(\bar{f}\) being injective, we may define a subspace topology \(\tau_\Sigma\) on \(X/{\sim}\) induced by \(u\bar{f}\).
  The composite \(u \bar{f} \colon X/{\sim} \to Y\) is continuous and \((X/{\sim},\tau_\Sigma)\) is Hausdorff.
  Notice that \(\sim\) satisfies, for every \(a,a',b,b' \in X\),
  \begin{equation*}
    a \sim a' \text{ and } b \sim b' \  \implies \ \Sigma \{a,b\} \sim \Sigma \{a',b'\}
  \end{equation*}
  because \(f\) is a \(\Sigma\)-homomorphism and \(X\) is finitely total:
  \begin{equation*}
    f(\Sigma \{a,b\}) = \Sigma \{f(a),f(b)\} = \Sigma \{f(a'),f(b')\} = f(\Sigma \{a',b'\})
  \end{equation*}
  Therefore, the operation \(+_\Sigma \colon X/{\sim} \times X/{\sim} \to X/{\sim}\) given by
  \begin{equation*}
    [a] +_\Sigma [b] = [\Sigma \{a,b\}]
  \end{equation*}
  is well-defined, and it is trivial to show that it forms a commutative monoid with the equivalence class \([0]\) as neutral element.
  Moreover, the composite \(u\bar{f} \colon X/{\sim} \to Y\) is a monoid homomorphism:
  \begin{equation*}
    u\bar{f}([a] +_\Sigma [b]) = f(\Sigma \{a,b\}) = \Sigma \{f(a),f(b)\} = f(a) + f(b) = u\bar{f}[a] + u\bar{f}[b]
  \end{equation*}
  where \(+\) is the monoid operation from \(Y\).
  To prove that \(+_\Sigma\) is continuous, first notice that for every two nets \(\alpha,\beta \colon D \to X\) the following chain of implications is satisfied:
  \begin{align*}
    \lim\, q\alpha &= [a] \text{ and } \lim\, q\beta = [b] && \\
      &\implies \lim f\alpha = f(a) \text{ and } \lim f\beta = f(b) &&\text{(\(u\bar{f}\) continuous, \(f = u\bar{f}q\))} \\
      &\implies \lim\, (f\alpha + f\beta) = f(a) + f(b) &&\text{(\(+\) is continuous)} \\
      &\implies \lim f(\Sigma \{\alpha,\beta\}) = f(\Sigma \{a,b\}) &&\text{(\(\Sigma\)-homomorphism, \(X \in \SCat{ft}\))} \\
      &\implies \lim u\bar{f}[\Sigma \{\alpha,\beta\}] = u\bar{f}[\Sigma \{a,b\}] &&\text{(\(f = u\bar{f}q\))} \\
      &\implies \lim\, [\Sigma \{\alpha,\beta\}] = [\Sigma \{a,b\}] &&\text{(\(\tau_\Sigma\) is the subspace topology)}
  \end{align*}
  Furthermore, notice that every net \(\alpha' \colon D \to X/{\sim}\) may be (non-uniquely) characterised by a net \(\alpha \colon D \to X\) such that \(\alpha' = q\alpha\), since \(q\) is surjective.
  Thus, according the chain of implications above, the function \(+_\Sigma\) is continuous and it follows that \((X/{\sim},\tau_\Sigma,+_\Sigma)\) is a Hausdorff commutative monoid.

  Next, we now show that \(q \colon X \to G(X/{\sim})\) is a \(\Sigma\)-homomorphism.
  Let \(\fml{x} \in X^*\) be a summable family and let \(\sigma_f \colon \fpset{f\fml{x}} \to Y\) and \(\sigma_q \colon \fpset{q\fml{x}} \to X/{\sim}\) be the nets of finite partial sums of \(f\fml{x} \in Y^*\) and \(q\fml{x} \in (X/{\sim})^*\) respectively.
  It is straightforward to check that \(u\bar{f}\sigma_q\) is a subnet of \(\sigma_f\).
  The following chain of implications establishes that \(q\) is a \(\Sigma\)-homomorphism:
  \begin{align*}
    \Sigma \fml{x} \keq x &\implies \Sigma f\fml{x} \keq f(x)  &&\text{(\(\Sigma\)-homomorphism \(f\))} \\
      &\implies \lim \sigma_f = f(x)  &&\text{(def.\@ \(\Sigma\) in \(G(Y)\))} \\
      &\implies \lim\, (u\bar{f} \circ \sigma_q) = u\bar{f}q(x)  &&\text{(\(u\bar{f}\sigma_q\) subnet, \(f = u\bar{f}q\))} \\
      &\implies \lim \sigma_q = q(x)  &&\text{(\(\tau_\Sigma\) is the subspace topology)} \\
      &\implies \Sigma q\fml{x} \keq q(x)  &&\text{(def.\@ \(\Sigma\) in \(G(X/{\sim})\))}
  \end{align*}

  In summary, \(X/{\sim}\) is a Hausdorff commutative monoid, and \(X \xto{q} G(X/{\sim})\) is an element of \(S\).
  Moreover, \(f = u \bar{f} \circ q\) and \(u\bar{f}\) is a continuous monoid homomorphism, providing a morphism \((q,X/{\sim}) \to (f,Y)\) in the comma category.
  This construction can be reproduced for every \(Y \in \HausCMon\) and every \(\Sigma\)-homomorphism \(f \colon X \to G(Y)\).
  Consequently, \(S\) is a weakly initial set in the comma category \(\comma{X}{G}\) and the claim that \(G\) has a left adjoint follows from the general adjoint functor theorem~\cite[Theorem~6.3.10]{Leinster}.
\end{proof}

\begin{corollary} \label{cor:HAG_g_adj}
  There is a left adjoint to the functor \(G \colon \HAG \to \SCat{g}\).
\end{corollary} \begin{proof}
  The claim follows from the same argument used in the previous theorem with the exception that, in this case, we must prove that \((X/{\sim},\tau_\Sigma,+_\Sigma)\) is a Hausdorff abelian group.
  It has already been established that \(X/{\sim} \in \HausCMon\), so we only need to prove that inverses exist in \(X/{\sim}\) for each element and that  inversion is continuous.
  Since in this case \(X\) is a \(\Sigma\)-group, the existence of inverses in \(X/{\sim}\) is immediate, where the inverse of every \([x] \in X/{\sim}\) is \([-x]\).
  To prove that inversion \([x] \mapsto [-x]\) is continuous, first notice that for every net \(\alpha \colon D \to X\) the following is satisfied:
  \begin{align*}
    \lim\, q\alpha = [a] &\implies \lim f\alpha = f(a) &&\text{(\(u\bar{f}\) continuous, \(f = u\bar{f}q\))} \\
      &\implies \lim\, (-f\alpha) = -f(a) &&\text{(continuous inversion map in \(Y\))} \\
      &\implies \lim f(-\alpha) = f(-a) &&\text{(group homomorphism \(f\))} \\
      &\implies \lim u\bar{f}[-\alpha] = u\bar{f}[-a] &&\text{(\(f = u\bar{f}q\))} \\
      &\implies \lim\, [-\alpha] = [-a] &&\text{(\(\tau_\Sigma\) is the subspace topology)}
  \end{align*}
  Recall that every net \(\alpha' \colon D \to X/{\sim}\) may be (non-uniquely) characterised by a net \(\alpha \colon D \to X\) such that \(\alpha' = q\alpha\), since \(q\) is surjective.
  Thus, the chain of implications above proves that inversion map in \(X/{\sim}\) is continuous and it follows that \((X/{\sim},\tau_\Sigma,+_\Sigma)\) is a Hausdorff abelian group. It then follows from the same argument as in the proof of Theorem~\ref{thm:HCMon_ft_adj} that \(q \colon X \to G(X/{\sim})\) is a \(\Sigma\)-homomorphism, providing all the tools necessary to show that \(\comma{X}{G}\) has a weakly initial set. It is straightforward to extend the proof Proposition~\ref{prop:HausCMon_complete} to show that \(\HAG\) is a complete category. It is immediate that \(\HAG\) is locally small and that \(G\) preserves limits. Consequently, the claim that \(G \colon \HAG \to \SCat{g}\) has a left adjoint follows from the general adjoint functor theorem~\cite[Theorem~6.3.10]{Leinster}.
\end{proof}

\section{Related work}
\label{sec:Sigma_rel_work}

The notion of strong \(\Sigma\)-monoid that is discussed in this document was originally proposed by Haghverdi~\cite{Haghverdi}, although similar mathematical structures predate it: for instance, the partially additive monoids of Manes and Arbib~\cite{ManesArbib} or the \(\Sigma\)-groups of Higgs~\cite{Higgs}.
We have proposed and studied a weaker version of \(\Sigma\)-monoids that admit additive inverses. Its axioms are inspired by the work of Higgs on \(\Sigma\)-groups~\cite{Higgs}.
We now provide some further discussion of a selection of works closely related to the contents of this paper.

\paragraph*{Haghverdi~\cite{Haghverdi}.} Haghverdi's (strong) \(\Sigma\)-monoids are defined in terms of two axioms: the unary sum axiom and the partition-asociativity axiom. The former corresponds precisely to the singleton axiom from weak \(\Sigma\)-monoids (Definition~\ref{def:wSm}) whereas the latter is equivalent to the combination of the axioms of subsummability, strong bracketing and strong flattening (Definition~\ref{def:sSm}).
Separating partition-associativity into its two directions of implication (\ie{} strong bracketing and strong flattening) allow us to weaken each of them in a different way.

\paragraph*{Hoshino~\cite{RTUDC}.} Hoshino proved that the category of total strong \(\Sigma\)-monoids (\ie{} the subcategory of \(\SCat{s}\) restricted to \(\Sigma\)-monoids whose \(\Sigma\) is a total function) is a reflective subcategory of \(\SCat{s}\).
This was achieved via a non-constructive argument using the general adjoint functor theorem and its strategy is reproduced in Appendix~\ref{sec:adjunctions} to prove that the embedding \(\SCat{ft} \into \SCat{w}\) has a left adjoint.
Interestingly, the proof of Proposition~\ref{prop:s_w_adj} --- which establishes that \(\SCat{s}\) is a reflective subcategory of \(\SCat{w}\) --- can be used (virtually without alteration) to give a constructive proof of Hoshino's result; in this case, the intersection construction from Lemma~\ref{lem:intersection_of_sSm} is no longer necessary.
Moreover, the proof of the existence of a monoidal structure given by the tensor product on every \(\SCat{*}\) category (Section~\ref{sec:tensor}) follows the same strategy Hoshino used in~\cite{RTUDC} to prove such a result for the particular case of \(\SCat{s}\).

\paragraph*{Higgs~\cite{Higgs}.} The \(\Sigma\)-groups defined in this paper are equivalent to Higgs' notion of \(\Sigma\)-groups; 
in particular, those where the set of summable families is restricted to countable families only.
The requirement that our families are countable is tied to our description of $\Sigma$-monoids as models of $\aleph_1$-ary essentially algebraic theories (Appendix~\ref{sec:locally_presentable}). However, the same strategy could be reproduced to characterise $\Sigma$-monoids as models of $\lambda$-ary essentially algebraic theories for other regular cardinals $\lambda$, so that our $\Sigma$-monoids may define summability of families of cardinality smaller than $\lambda$. This would weaken the claim of Proposition~\ref{prop:preserve_limits_chain_colimits} to embedding functors between $\SCat{*}$ categories preserving $\lambda$-directed colimits, which remains sufficient for the proof of Theorem~\ref{thm:embedding_right_adjoint}.

To verify that the axioms in Higgs' definition of \(\Sigma\)-groups~\cite{Higgs} coincides with ours, notice that the neutral element axiom and the singleton axiom of weak \(\Sigma\)-monoids appear explicitly in Higgs' definition as axioms (\(\Sigma1\)) and (\(\Sigma2\)) respectively, whereas finite totality is imposed by Higgs by requiring that the \(\Sigma\)-group contains an abelian group.
The bracketing axiom of weak \(\Sigma\)-monoids is Higgs' (\(\Sigma4\)) axiom, whereas Higgs' (\(\Sigma3\)) axiom corresponds to flattening and is stated as follows (paraphrased using our notation):
\begin{equation*}
  \Sigma \fml{x} \keq x \text{ and } \Sigma \fml{x'} \keq x' \implies \Sigma (\fml{x} \uplus -\fml{x'}) \keq x - x'.
\end{equation*}
Since flattening in weak \(\Sigma\)-monoids only discusses addition of finite collection of families and since \(\Sigma\)-groups are required to be finitely total, Higgs' axiom implies flattening.
Moreover, when \(\fml{x} = \emptyset\), Higgs' axiom implies that every summable family \(\fml{x'}\) has a counterpart \(-\fml{x'}\) so that \(\Sigma (-\fml{x'}) \keq -(\Sigma \fml{x'})\).
Thus, Higgs combines in a single axiom both the flattening axiom of weak \(\Sigma\)-monoids and the requirement that the inversion map is a \(\Sigma\)-homomorphism.
With these remarks in mind, it is straightforward to check that Higgs' definition of \(\Sigma\)-groups is equivalent to the definition presented in this paper.

On another note, the realisation that Hausdorff abelian groups can be captured as instances of \(\Sigma\)-groups is briefly mentioned in the introduction of Higgs' paper~\cite{Higgs}, although no explicit proof is provided.
It was Higgs' definition that revealed the subtle weakening to bracketing that enables us to define a functor \(\HAG \to \SCat{g}\).
In contrast, Higgs proved that every \(\Sigma\)-group is a net group and, moreover, the corresponding embedding functor has a right adjoint.
This is somewhat dual to the result from Section~\ref{sec:HausCMon->SCatft} where the functor \(\HAG \to \SCat{g}\) is shown to have a left adjoint.

\paragraph*{Tsukada and Asada~\cite{tsukadaasada:linearlogic}}
Tsukada and Asada follow Haghverdi's definition of $\Sigma$-monoid, phrasing the axioms in an alternative way. They also consider $\Sigma$-semirings --- which are roughly $\Sigma$-monoids furnished with a multiplication that distributes over the sum --- and modules over $\Sigma$-semirings. As another application of the use of $\Sigma$-monoids, they prove that the category of modules uniformly captures different known models for linear logic by varying the $\Sigma$-semiring.

\bibliographystyle{plain}
\bibliography{Bibliography}

\appendix

\section{Every \(\SCat{*}\) category is locally presentable}
\label{sec:locally_presentable}

The axioms of \(\Sigma\)-monoids can be formalised in the language of essentially algebraic theories.
It is well known~\cite[Theorem 3.36]{Adamek_Rosicky} that the category of models of any such theory is locally presentable, thus providing us with a straightforward path to show that every \(\SCat{*}\) category is locally presentable.

A (single-sorted) signature is a set of symbols where each symbol has a cardinal number assigned to it, indicating its arity.
We use the notation \(\sigma \colon \bullet^n \to \bullet\) to refer to an \(n\)-ary symbol and refer to symbols with arity \(0\) as constants.
Given a signature \(\mathcal{S}\), a term is either a standard variable \(x_i\) or an \(n\)-ary symbol \(\sigma \in \mathcal{S}\) along with \(n\) input terms \(t_i\):
\begin{equation*}
  \sigma \left( \prod_{i \leq n} t_i \right)
\end{equation*}
An equation is an expression of the form \(t = t'\) where \(t\) and \(t'\) are terms on a common signature.

The definition of essentially algebraic theory used here corresponds to Definition 3.34 from~\cite{Adamek_Rosicky} with the generalisation proposed in Remark 3.37 from the same source, and restricted to the single-sorted case.

\begin{definition}
A (single-sorted) \emph{essentially algebraic theory} is a tuple \((\mathcal{S}, \leq, E, \Def)\) consisting of a (single-sorted) signature \(\mathcal{S}\) with a well-ordering \(\leq\), a set of equations \(E\) on \(\mathcal{S}\), and a function \(\Def\) assigning to each symbol \(\sigma \in \mathcal{S}\) a set \(\Def(\sigma)\) of equations on symbols from \(\{\tau \in \mathcal{S} \mid \tau < \sigma\}\).
\end{definition}

We say that an essentially algebraic theory is \(\lambda\)-ary, for a regular cardinal \(\lambda\), provided all of the symbols in its signature have arity strictly smaller than \(\lambda\), for each \(\sigma \in \mathcal{S}\) the cardinal of the set \(\Def(\sigma)\) is strictly smaller than \(\lambda\), and each of the equations in \(E\) and \(\Def(\sigma)\) use fewer than \(\lambda\) standard variables.

\begin{definition}
  A \emph{model} of an essentially algebraic theory \((\mathcal{S}, \leq, E, \Def)\) is a set \(A\) together with partial functions \(\sigma_A \colon \prod_{i \leq n} A \pto A\) for each symbol \(\sigma \colon \bullet^n \to \bullet\) from the signature \(\mathcal{S}\) such that:
  \begin{itemize}
    \item \(\sigma_A\) is defined on inputs \((a_i)_{i \leq n}\) if and only if all of the equations in \(\Def(\sigma)\) hold after substitution of variables \(x_i \mapsto a_i\) and symbols \(\tau \mapsto \tau_A\),
    \item and for every assignment of variables \(x_i \mapsto a_i\) and every equation in \(E\), if both sides of the equation are well-defined after substitution, then their results match.
  \end{itemize}
  A \emph{homomorphism} between models is a function \(f \colon A \to B\) between their underlying sets such 
  \begin{equation*}
    \sigma_A \left(\prod_{i \leq n} a_i \right) \keq a' \implies \sigma_B \left(\prod_{i \leq n} f(a_i) \right) \keq f(a').
  \end{equation*}
  where $\keq$ is the Kleene equality of partial functions, as defined in Notation~\ref{not:keq}.
  The category of models of an essentially algebraic theory is comprised of all of the models of the theory and the homomorphisms between them.
\end{definition}

We first construct an essentially algebraic theory for weak \(\Sigma\)-monoids. We do so incrementally, including new symbols and equations as necessary to capture each of the axioms of weak \(\Sigma\)-monoids (Definition~\ref{def:wSm}).
The well-ordering in \(\mathcal{S}\) satisfies the chronological order in which we introduce the symbols in our discussion; the particular well-ordering among the symbols introduced in a single step can be chosen arbitrarily.
We use a single-sorted signature \(\mathcal{S}\) with a constant symbol \(0\) and, for each ordinal \(n \leq \omega\), an \(n\)-ary symbol \(s_n \colon \bullet^n \to \bullet\). In particular, \(s_\omega\) has countably infinite arity.
These symbols are total, \ie{} \(\Def(0) = \emptyset\) and \(\Def(s_n) = \emptyset\) for all \(n \leq \omega\).
For our first subset of partial symbols, we introduce an \(n\)-ary symbol \(\sigma_n \colon \bullet^n \to \bullet\) for each \(n \leq \omega\) with domain of definition
\begin{equation*}
  \Def(\sigma_n) = \left\{s_n \left(\prod_{i \leq n} x_i \right) = 0 \right\}.
\end{equation*}
Intuitively, \(s_n(\ldots) = 0\) identifies the input family of \(n\) elements as summable and, if so, \(\sigma_n\) provides the result of the sum.
Since \(\Sigma\)-monoids deal with families rather than sequences, we need to add equations for arbitrary permutations of the inputs:
\begin{equation*}
  \left\{s_n \left(\prod_{i \leq n} x_i \right) = s_n \left(\prod_{i \leq n} x_{p(i)} \right) \Mid n \leq \omega, p \in S(n) \right\} \subset E
\end{equation*}
\begin{equation*}
  \left\{\sigma_n \left(\prod_{i \leq n} x_i \right) = \sigma_n \left(\prod_{i \leq n} x_{p(i)} \right) \Mid n \leq \omega, p \in S(n) \right\} \subset E
\end{equation*}
where \(S(n)\) is the symmetric group.

The singleton axiom is captured by the following equations:
\begin{equation*}
  \{s_1(x) = 0,\ \sigma_1(x) = x\} \subset E.
\end{equation*}
For the neutral element axiom, we include the following equations:
\begin{equation*}
  \left\{s_{n+z} \left(\prod_{i \leq n} x_i, \prod_{i \leq z} 0\right) = s_n \left(\prod_{i \leq n} x_i\right) \Mid n,z \leq \omega \right\} \subset E
\end{equation*}
\begin{equation*}
  \left\{\sigma_{n+z} \left(\prod_{i \leq n} x_i, \prod_{i \leq z} 0\right) = \sigma_n \left(\prod_{i \leq n} x_i\right) \Mid n,z \leq \omega \right\} \subset E
\end{equation*}
This differs slightly from the neutral element axiom of Definition~\ref{def:wSm}, which discusses summability of families but requires nothing about their results, so the equations on $\sigma$ might seem to impose stronger restrictions on the models of this theory.
Notice, however, that these equations on $\sigma$ are satisfied by any weak $\Sigma$-monoid, as established in Proposition~\ref{prop:neutral_element}.
The reason we include these equations explicitly is that the formalisation of the bracketing and flattening axioms we present below does not consider empty subfamilies, so the proof of the latter proposition cannot be reproduced for models of this theory and we must instead impose the result explicitly.
On the other hand, it is trivial to verify that \(s_0 = 0\) and \(\sigma_0 = 0\) follow from the equations introduced so far.

To capture the implication in the bracketing axiom we use the approach from Example 3.35 (3) in~\cite{Adamek_Rosicky} and introduce new symbols whose domain of definition correspond to the conditions of the implication.
The bracketing axiom requires us to discuss partitions of families, so for each \(n \leq \omega\), we consider the set \(\mathcal{F}_n\) of all functions \(f \colon n+1 \to \omega + 1\) --- here we treat each ordinal \(\lambda\) as the well-ordered set of all smaller ordinals \([0,\lambda)\).
For each such function \(f \in \mathcal{F}_n\) we define an auxiliary set of equations (not included in \(E\)).
\begin{equation*}
  \Delta_f = \left\{ s_{\lvert f^{-1}(j) \rvert} \left(\prod_{i \in f^{-1}(j)} x_i \right) = 0 \Mid j \in \im(f) \right\}
\end{equation*}
These capture the requirement that each family in the partition is summable.
The new symbols introduced for bracketing are:
\begin{equation*}
  \left\{\beta^s_{n,f} \colon \bullet^n \to \bullet \Mid n \leq \omega,\, f \in \mathcal{F}_n,\, j \in \im(f),\, \lvert f^{-1}(j) \rvert < \omega \right\} \subset \mathcal{S}
\end{equation*}
\begin{equation*}
  \left\{\beta^\sigma_{n,f} \colon \bullet^n \to \bullet \Mid n \leq \omega,\, f \in \mathcal{F}_n,\, j \in \im(f),\, \lvert f^{-1}(j) \rvert < \omega \right\} \subset \mathcal{S}
\end{equation*}
where we have imposed that each family in the partition is finite.
The domain of definition of these symbols is
\begin{equation*}
  \Def(\beta^\sigma_{n,f}) = \Def(\beta^s_{n,f}) = \Delta_f \cup
    \left\{ s_n \left(\prod_{i \leq n} x_i \right) = 0 \right\}
\end{equation*}
and we introduce the following equations for each \(\beta^s_{n,f}, \beta^\sigma_{n,f} \in \mathcal{S}\):
\begin{equation*}
  \left\{ \beta^s_{n,f} \left(\prod_{i \leq n} x_i \right) = s_n \left(\prod_{i \leq n} x_i \right),\ \beta^\sigma_{n,f} \left(\prod_{i \leq n} x_i \right) = \sigma_n \left(\prod_{i \leq n} x_i \right) \right\} \subset E
\end{equation*}
\begin{equation*}
  \left\{ \beta^s_{n,f} \left(\prod_{i \leq n} x_i \right) = s_{\lvert \im(f) \rvert} \left( \prod_{j \in \im(f)} \sigma_{\lvert f^{-1}(j) \rvert} \left(\prod_{i \in f^{-1}(j)} x_i \right) \right) \right\} \subset E
\end{equation*}
\begin{equation*}
  \left\{ \beta^\sigma_{n,f} \left(\prod_{i \leq n} x_i \right) = \sigma_{\lvert \im(f) \rvert} \left( \prod_{j \in \im(f)} \sigma_{\lvert f^{-1}(j) \rvert} \left(\prod_{i \in f^{-1}(j)} x_i \right) \right) \right\} \subset E
\end{equation*}
which, when \(\beta^s_{n,f}\) is defined, imply (by transitiviy of \(=\)) the right-hand side of the bracketing axiom.

The flattening axiom is captured in a similar way.
We introduce the following symbols:
\begin{equation*}
  \left\{\varphi^s_{n,f} \colon \bullet^n \to \bullet \Mid n \leq \omega,\, f \in \mathcal{F}_n,\, \lvert \im(f) \rvert < \omega \right\} \subset \mathcal{S}
\end{equation*}
\begin{equation*}
  \left\{\varphi^\sigma_{n,f} \colon \bullet^n \to \bullet \Mid n \leq \omega,\, f \in \mathcal{F}_n,\, \lvert \im(f) \rvert < \omega \right\} \subset \mathcal{S}
\end{equation*}
where we have imposed that the number of families in the partition is finite.
The domain of definition of these symbols is
\begin{equation*}
  \Def(\varphi^s_{n,f}) = \Delta_f \cup
    \left\{ s_{\lvert \im(f) \rvert} \left( \prod_{j \in \im(f)} \sigma_{\lvert f^{-1}(j) \rvert} \left(\prod_{i \in f^{-1}(j)} x_i \right) \right) = 0 \right\}
\end{equation*}
\begin{equation*}
  \Def(\varphi^\sigma_{n,f}) = \Def(\varphi^s_{n,f})
\end{equation*}
and the equations
\begin{equation*}
  \left\{ \varphi^s_{n,f} \left(\prod_{i \leq n} x_i \right) = s_n \left(\prod_{i \leq n} x_i \right),\ \varphi^\sigma_{n,f} \left(\prod_{i \leq n} x_i \right) = \sigma_n \left(\prod_{i \leq n} x_i \right) \right\} \subset E
\end{equation*}
\begin{equation*}
  \left\{ \varphi^s_{n,f} \left(\prod_{i \leq n} x_i \right) = s_{\lvert \im(f) \rvert} \left( \prod_{j \in \im(f)} \sigma_{\lvert f^{-1}(j) \rvert} \left(\prod_{i \in f^{-1}(j)} x_i \right) \right) \right\} \subset E
\end{equation*}
\begin{equation*}
  \left\{ \varphi^\sigma_{n,f} \left(\prod_{i \leq n} x_i \right) = \sigma_{\lvert \im(f) \rvert} \left( \prod_{j \in \im(f)} \sigma_{\lvert f^{-1}(j) \rvert} \left(\prod_{i \in f^{-1}(j)} x_i \right) \right) \right\} \subset E
\end{equation*}
imply the right hand side of the flattening axiom when \(\varphi^s_{n,f}\) is defined.

This theory is \(\aleph_1\)-ary and, by construction, there is a one-to-one correspondence between models of this theory and weak \(\Sigma\)-monoids. Morphisms in the category of models are precisely the \(\Sigma\)-homomorphisms in \(\SCat{w}\). Consequently, the category of models of the essentially algebraic theory presented above is equivalent to \(\SCat{w}\).

We now sketch what needs to be added to the above theory so that its category of models corresponds to each of the other \(\SCat{*}\) categories:
\begin{itemize}
  \item For \(\SCat{s}\) we introduce extra symbols \(\beta\) and \(\varphi\) (along with their equations) for all \(n \leq \omega\) and \(f \in \mathcal{F}_n\) that were not included already, so that we are lifting the constraints on finiteness from bracketing and flattening. The subsummability axiom is captured with a new set of symbols \(\tau_{n,U}\) for each \(n \leq \omega\) and each subset \(U \subset n+1\). The domain of definition of each \(\tau_{n,U}\) is determined by \(s_n(\ldots) = 0\) and new equations are added to \(E\) to enforce that, when \(\tau_{n,U}\) is defined, \(s_{\lvert U \rvert}(\ldots) = 0\).
  \item For \(\SCat{ft}\) we instead include in \(E\) an equation \(s_n(\ldots) = 0\) for each finite ordinal \(n < \omega\).
  \item For \(\SCat{g}\) we introduce to the theory of \(\SCat{ft}\) a total symbol \(\iota \colon \bullet \to \bullet\) to capture the inversion map and include equations
  \begin{equation*}
    \{s_2(x, \iota(x)) = 0,\ \sigma_2(x, \iota(x)) = 0\} \subset E
  \end{equation*}
  so that every element has an inverse and
  \begin{equation*}
    \left\{s_n \left( \prod_{i \leq n} \iota(x_i) \right) = s_n \left( \prod_{i \leq n} x_i \right) \Mid n \leq \omega \right\} \subset E
  \end{equation*}
  \begin{equation*}
    \left\{\sigma_n \left( \prod_{i \leq n} \iota(x_i) \right) = \iota \left( \sigma_n \left( \prod_{i \leq n} x_i \right)\right) \Mid n \leq \omega \right\} \subset E
  \end{equation*}
  so that the inverse map is a \(\Sigma\)-homomorphism.
\end{itemize}

\begin{corollary} \label{cor:SCat_locally_presentable}
  Each of the \(\SCat{*}\) categories is locally \(\aleph_1\)-presentable.
\end{corollary}
\begin{proof}
  We have shown that every \(\SCat{*}\) category is equivalent to the category of models of some \(\aleph_1\)-ary essentially algebraic theory. Theorem 3.6 of~\cite{Adamek_Rosicky} establishes that such categories of models are locally \(\aleph_1\)-presentable.
\end{proof}

\section{Construction of the left adjoint to \(\SCat{s} \into \SCat{w}\)}
\label{sec:adjunctions}

The left adjoint to the embedding \(\SCat{s} \into \SCat{w}\) is constructed explicitly in Proposition~\ref{prop:s_w_adj}.
The key insight is that the strong versions of bracketing and flattening can be recovered by defining a  congruence \(\sim\) and quotienting the set of all families with respect to it.
Then, we ought to find the `smallest' strong \(\Sigma\)-monoid defined on this quotient set that satisfies the universal property of the adjunction.
The following lemma is instrumental to find such a `smallest' strong \(\Sigma\)-monoid.

\begin{lemma} \label{lem:intersection_of_sSm}
  Let \(\{(X_i,\Sigma^i)\}_I\) be a collection of strong \(\Sigma\)-monoids.
  Define a partial function \(\Sigma^\cap \colon (\cap_I X_i)^* \pto \cap_I X_i\) as follows:
  \begin{equation*}
    \Sigma^\cap \fml{x} = \begin{cases}
      x &\ifc \forall i \in I,\ \Sigma^i \fml{x} \keq x \\
      \undefined &\otherwise.
    \end{cases}
  \end{equation*}
  Then, \((\cap_I X_i, \Sigma^\cap)\) is a strong \(\Sigma\)-monoid.
\end{lemma} \begin{proof}
  Whenever a family \(\fml{x}\) is summable in \(\Sigma^\cap\), it must be summable in \(\Sigma^i\) for all \(i \in I\) and, considering that each \((X_i,\Sigma^i)\) is a strong \(\Sigma\)-monoid, it follows that every subfamily of \(\fml{x}\) is summable in every \(\Sigma^i\) and, hence, it is summable in \(\Sigma^\cap\) so that subsummability is satisfied.
  Similarly, strong bracketing and strong flattening in \(\Sigma^\cap\) follow from those in each \(\Sigma^i\) and the neutral element and singleton axioms are trivially satisfied.
  Consequently, \((\cap_I X_i, \Sigma^\cap)\) is a strong \(\Sigma\)-monoid.
\end{proof}

\begin{proposition} \label{prop:s_w_adj}
  There is a left adjoint functor to the embedding \(\SCat{s} \into \SCat{w}\).
\end{proposition} \begin{proof}
  Let \(G \colon \SCat{s} \into \SCat{w}\) denote the embedding; its left adjoint \(F \colon \SCat{w} \to \SCat{s}\) is defined explicitly.
  Let \(X \in \SCat{w}\) be an arbitrary weak \(\Sigma\)-monoid and define a relation \(\leadsto\) on \(X^*\) such that for any two families \(\fml{x},\fml{x'} \in X^*\), we have that \(\fml{x} \leadsto \fml{x'}\) iff there is a partition of \(\fml{x}\) into subfamilies \(\fml{x} = \uplus_J \fml{x}_j\) such that each (possibly infinite) subfamily \(\fml{x}_j\) is summable and such that \(\fml{x'}\) is the family of their sums \(\fml{x'} = \{\Sigma \fml{x}_j\}_J\).
  The relation \(\leadsto\) is reflexive since any family may be partitioned into its singleton subfamilies.
  Let \(\sim\) be the equivalence closure of \(\leadsto\); then, \(\fml{x} \sim \fml{x'}\) iff there is a zig-zag chain of \(\leadsto\) relations such as:
  
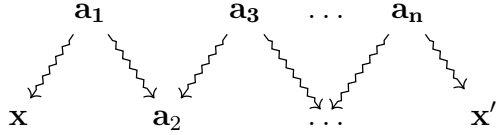
\[\begin{tikzcd}[column sep=tiny]
    & {\fml{a_1}} && {\fml{a_3}} & \ldots & {\fml{a_n}} \\
    {\fml{x}} && {\fml{a}_2} && \ldots && {\fml{x'}}
    \arrow[squiggly, from=1-2, to=2-1]
    \arrow[squiggly, from=1-2, to=2-3]
    \arrow[squiggly, from=1-4, to=2-3]
    \arrow[squiggly, from=1-4, to=2-5]
    \arrow[squiggly, from=1-6, to=2-5]
    \arrow[squiggly, from=1-6, to=2-7]
  \end{tikzcd}\]
  Notice that \(\sim\) is a congruence with respect to arbitrary disjoint union, \ie{} it satisfies:
  \begin{equation} \label{eq:weak_strong_sim_cong}
    \forall j \in J,\, \fml{x}_j \sim \fml{x'}_j \implies \uplus_J \fml{x}_j \sim \uplus_J \fml{x'}_j
  \end{equation}
  since a zig-zag chain from \(\uplus_J \fml{x}_j\) to \(\uplus_J \fml{x'}_j\) may be obtained by composing the chains relating each subfamily \(\fml{x}_j\) to \(\fml{x'}_j\).
  Let \(A\) be the quotient set \(X^*/{\sim}\) and let \(q \colon X \to A\) be the function that maps each \(x \in X\) to the equivalence class \([\{x\}] \in A\).

  Let \(Y \in \SCat{s}\) be an arbitrary strong \(\Sigma\)-monoid and let \(f \colon X \to G(Y)\) be an arbitrary \(\Sigma\)-homomorphism.
  Define \(\leadsto\) on \(Y^*\) as above; since \(f\) is a \(\Sigma\)-homomorphism we have that \(\fml{x} \leadsto \fml{x'}\) implies \(f\fml{x} \leadsto f\fml{x'}\) and, consequently,
  \begin{equation} \label{eq:sim_hom}
    \fml{x} \sim \fml{x'} \implies f\fml{x} \sim f\fml{x'}.
  \end{equation}
  Moreover, in a strong \(\Sigma\)-monoid \(\fml{y} \leadsto \fml{y'}\) implies \(\Sigma \fml{y} \keq \Sigma \fml{y'}\), since if \(\fml{y}\) is summable then \(\fml{y'}\) is summable due to strong bracketing whereas if \(\fml{y'}\) is summable then \(\fml{y}\) is summable due to strong flattening.
  Consequently, we have that for any two families \(\fml{x},\fml{x'} \in X^*\):
  \begin{equation} \label{eq:s_w_sim_implies}
    \fml{x} \sim \fml{x'} \implies \Sigma f\fml{x} \keq \Sigma f\fml{x'}.
  \end{equation}
  Let \(A_f\) be the subset of \(A\) defined as follows:
  \begin{equation*}
    A_f = \{[\fml{x}] \in A \mid f\fml{x} \text{ is summable in } Y\}.
  \end{equation*}
  Notice that, for any equivalence class \([\fml{x}] \in A_f\), each of its families \(\fml{x'} \in [\fml{x}]\) satisfies that \(f\fml{x'}\) is summable, due to the implication~\eqref{eq:s_w_sim_implies} and the fact that \(f\fml{x}\) is summable.
  Let \(\Sigma^f \colon A_f^* \pto A_f\) be the following partial function:
  \begin{equation*}
    \Sigma^f \{[\fml{x}_i]\}_I = \begin{cases}
      [\uplus_I \fml{x}_i] &\ifc [\uplus_I \fml{x}_i] \in A_f \\
      \undefined &\otherwise.
    \end{cases}
  \end{equation*}
  Notice that this definition is independent of the choice of representatives thanks to \(\sim\) being a congruence~\eqref{eq:weak_strong_sim_cong}.
  We must check that \((A_f,\Sigma^f)\) is a strong \(\Sigma\)-monoid.
  \begin{description}
    \item[Singleton.] Singleton families are trivially summable.
    \item[Subsummability.] If \(\{[\fml{x}_i]\}_I \in A_f^*\) is summable then \(\uplus_I f\fml{x}_i\) is summable in \(Y\) and so is \(\uplus_J f\fml{x}_j\) for any \(J \subseteq I\) due to subsummability in \(Y\). Thus, \(\{[\fml{x}_j]\}_J\) is summable in \(A_f\) and \(\Sigma^f\) satisfies subsummability.
    \item[Neutral element.] Follows immediately from subsummability.
    \item[Strong flattening.] Let \(I\) be an arbitrary set partitioned into \(I = \uplus_J I_j\) and, for each \(j \in J\), let \(\{[\fml{x}_i]\}_{I_j}\) be a summable family in \(A_f\), \ie{} \(\Sigma^f \{[\fml{x}_i]\}_{I_j} \keq [\uplus_{I_j} \fml{x}_i]\). Assume that the family \(\{[\uplus_{I_j} \fml{x}_i]\}_J \in A_f^*\) is summable, \ie{} \(\Sigma^f \{[\uplus_{I_j} \fml{x}_i]\}_J \keq [\uplus_J (\uplus_{I_j} \fml{x}_i)]\); then, due to associativity and commutativity of \(\uplus\) we have that \(\uplus_I \fml{x}_i = \uplus_J (\uplus_{I_j} \fml{x}_i)\) and, hence, \([\uplus_I \fml{x}_i] \in A_f\) so that \(\{[\fml{x}_i]\}_I\) is summable in \(A_f\).
    \item[Strong bracketing.] As in the case of flattening, this axiom follows from associativity and commutativity of disjoint union.
  \end{description}

  Let \(S\) be the set of all such strong \(\Sigma\)-monoids \((A_f,\Sigma^f)\).\footnote{Even though \(f \colon X \to G(Y)\) ranges over all \(Y \in \SCat{s}\), each \(A_f\) is by definition a subset of \(A\) which is small since \(X^*\) is small (see Proposition~\ref{prop:family_set}), hence, \(S\) is a (small) set.}
  Let \(F(X)\) be the strong \(\Sigma\)-monoid obtained as the intersection of all the members of \(S\) as per Lemma~\ref{lem:intersection_of_sSm}.
  Let \(\eta_X \colon X \to GF(X)\) be the function that maps each \(x \in X\) to \([\{x\}]\); such an equivalence class is present in every \(A_f\), so it is present in \(F(X)\).
  Notice that \(\eta_X\) is a \(\Sigma\)-homomorphism: if \(\fml{x} = \{x_i\}_I \in X^*\) is summable, then \(\fml{x} \leadsto \{\Sigma \fml{x}\}\) and \([\fml{x}] = [\{\Sigma \fml{x}\}]\), hence,
  \begin{equation*}
    \Sigma^\cap \{\eta_X(x_i)\}_I \keq [\uplus_I \{x_i\}] = [\fml{x}] = [\{\Sigma \fml{x}\}] = \eta_X(\Sigma \fml{x}).
  \end{equation*}
  On morphisms \(f \colon X \to X'\), let \(F(f)\) be the function that maps \([\fml{x}] \in F(X)\) to \([f\fml{x}] \in F(X')\); this is a \(\Sigma\)-homomorphism thanks to~\eqref{eq:sim_hom}.
  It is straightforward to check that this construction yields a functor \(F \colon \SCat{w} \to \SCat{s}\) and a natural transformation \(\eta\) whose components \(\eta_X\) were defined above.

  It remains to check that \(F\) is left adjoint to \(G\).
  Fix some arbitrary \(X \in \SCat{w}\), \(Y \in \SCat{s}\) and \(f \in \SCat{w}(X,G(Y))\).
  There is a unique \(\Sigma\)-homomorphism \(\bar{f} \colon F(X) \to Y\) making the diagram
  \[\begin{tikzcd}
    X && {G(Y)} \\
    && {GF(X)}
    \arrow["f", from=1-1, to=1-3]
    \arrow["{\eta_X}"', from=1-1, to=2-3]
    \arrow["{G(\bar{f})}"', dashed, from=2-3, to=1-3]
  \end{tikzcd}\]
  commute: the requirement that \(f = G(\bar{f}) \circ \eta_X\) imposes that \(\bar{f}\) maps each \(\eta_X(x) = [\{x\}]\) to \(f(x)\); the requirement that \(\bar{f}\) is a \(\Sigma\)-homomorphism imposes that, for any arbitrary equivalence class \([\{x_i\}_I] \in F(X)\),
  \begin{equation*}
    \bar{f} [\{x_i\}_I] = \bar{f} [\uplus_I \{x_i\}] = \bar{f} (\Sigma^\cap \{\eta_X(x_i)\}_I) = \Sigma \{f(x_i)\}_I.
  \end{equation*}
  The value of such a \(\Sigma\)-homomorphism \(\bar{f}\) is uniquely determined and, thus, it has been shown that \(F\) is left adjoint to \(G\), as claimed.
\end{proof}

\section{Deferred proofs}

\subsection{Supporting proofs for Section~\ref{sec:tensor}}
\label{sec:tensor_appendix}

\begin{lemma} \label{lem:bihom_set_functor_preserves_limits}
  For all \(X,Y \in \SCat{w}\), the functor \(\SCat{w}^{X,Y} \colon \SCat{w} \to \Set\) preserves limits.
\end{lemma} \begin{proof}
  Let \(Z,W \in \SCat{w}\) be arbitrary objects, let \(f,g \in \SCat{w}(Z,W)\) be arbitrary \(\Sigma\)-homomorphisms and let \(E \in \SCat{w}\) be their equaliser constructed as in Proposition~\ref{prop:preserve_limits_chain_colimits}. Recall that \(E \subseteq Z\) and \(z \in E\) implies \(f(z) = g(z)\) so it is clear that any \(h \in \SCat{w}^{X,Y}(E)\) satisfies \(f \circ h = g \circ h\).
  Conversely, if a \(\Sigma\)-bilinear function \(h \colon X \times Y \to Z\) satisfies \(f \circ h = g \circ h\) then, for every element \(z \in \im(h)\), it holds that \(f(z) = g(z)\), so \(\im(h) \subseteq E\).
  Therefore, \(\SCat{w}^{X,Y}(E)\) is precisely the subset of \(\SCat{w}^{X,Y}(Z)\) such that for all \(h \in \SCat{w}^{X,Y}(E)\),
  \begin{equation*}
    \SCat{w}^{X,Y}(f)(h) = f \circ h = g \circ h = \SCat{w}^{X,Y}(g)(h)
  \end{equation*}
  \ie{} \(\SCat{w}^{X,Y}(E)\) is an equaliser of the diagram in \(\Set\).
  Thus, \(\SCat{w}^{X,Y}\) preserves equalisers.

  Let \(Z \times W \in \SCat{w}\) be the categorical product of two weak \(\Sigma\)-monoids and let \(\pi_l\) and \(\pi_r\) be the corresponding projections.
  Since the categorical product in \(\Set\) is given by the Cartesian product, there is a unique function \(m\) making the following diagram
  \[\begin{tikzcd}[column sep=small]
    && {\SCat{w}^{X,Y}(Z \times W)} \\
    \\
    {\SCat{w}^{X,Y}(Z)} && {\SCat{w}^{X,Y}(Z) \times \SCat{w}^{X,Y}(W)} && {\SCat{w}^{X,Y}(W)}
    \arrow["{\SCat{w}^{X,Y}(\pi_l)}"', from=1-3, to=3-1]
    \arrow["{\pi_l}", from=3-3, to=3-1]
    \arrow["m", dashed, from=1-3, to=3-3]
    \arrow["{\SCat{w}^{X,Y}(\pi_r)}", from=1-3, to=3-5]
    \arrow["{\pi_r}"', from=3-3, to=3-5]
  \end{tikzcd}\]
  commute in \(\Set\).
  Such a function \(m\) maps every \(\Sigma\)-bilinear function \(f \in \SCat{w}^{X,Y}(Z \times W)\) to the pair of \(\Sigma\)-bilinear functions \((\pi_l \circ f, \pi_r \circ f)\).
  Conversely, there is a function \(u \colon \SCat{w}^{X,Y}(Z) \times \SCat{w}^{X,Y}(W) \to \SCat{w}^{X,Y}(Z \times W)\) mapping each pair of \(\Sigma\)-bilinear functions \((g,h)\) to the function given below:
  \begin{equation*}
    k(x,y) = (g(x,y),h(x,y))
  \end{equation*}
  for every \(x \in X\) and \(y \in Y\).
  Notice that \(u\) is well-defined since \(k\) is a \(\Sigma\)-bilinear function:
  \begin{align*}
    k(\Sigma \fml{x},y) &= (g(\Sigma \fml{x},y),h(\Sigma \fml{x},y)) = (\Sigma g(\fml{x},y),\Sigma h(\fml{x},y)) \\
      &= \Sigma^\times (g(\fml{x},y), h(\fml{x},y)) = \Sigma^\times k(\fml{x},y)
  \end{align*}
  where \(g(\fml{x},y)\) is a shorthand for the family obtained after applying \(g(-,y)\) to each element in \(\fml{x}\).
  It is straightforward to check that \(u \circ m = \id\) and \(m \circ u = \id\) so that \(\SCat{w}^{X,Y}(Z \times W)\) is isomorphic to \(\SCat{w}^{X,Y}(Z) \times \SCat{w}^{X,Y}(W)\).
  Thus, it follows that \(\SCat{w}^{X,Y}(Z \times W)\) is a categorical product and \(\SCat{w}^{X,Y}\) preserves binary products.
  It is straightforward to generalise this argument to small products and, hence, \(\SCat{w}^{X,Y}\) preserves small products.

  A functor that preserves equalisers and small products and whose domain is a complete category automatically preserves all limits. Therefore, \(\SCat{w}^{X,Y}\) preserves limits, as claimed.
\end{proof}

\begin{lemma} \label{lem:SCatw_tensor}
 The comma category \(\comma{\{\bullet\}}{\SCat{w}^{X,Y}}\) has an initial object. Such an initial object is denoted \((p,X \otimes Y)\).
\end{lemma} \begin{proof}
  The category \(\SCat{w}\) is complete (see Proposition~\ref{prop:SCat*_locally_presentable}) and
  Lemma~\ref{lem:bihom_set_functor_preserves_limits} establishes that the functor \(\SCat{w}^{X,Y}\) preserves limits.
  Consequently, the comma category \(\comma{\{\bullet\}}{\SCat{w}^{X,Y}}\) is complete~\cite[Lemma A.2]{Leinster} and, since \(\SCat{w}\) is locally small, it follows that the comma category is locally small.
  Then, it suffices to provide a weakly initial set to prove that the comma category \(\comma{\{\bullet\}}{\SCat{w}^{X,Y}}\) has an initial object~\cite[Lemma A.1]{Leinster}.
  The elements of such a weakly initial set ought to be functions of type \(\{\bullet\} \to \SCat{w}^{X,Y}(A)\) for some \(A \in \SCat{w}\); but a function with singleton domain simply selects an element in its codomain, so it is equivalent to think of the weakly initial set as a collection of \(\Sigma\)-bilinear functions \(X \times Y \xto{q} A\).
  Let \(S\) be the collection of all \(\Sigma\)-bilinear functions with domain \(X \times Y\) and whose codomain may be any weak \(\Sigma\)-monoid whose underlying set is a quotient of a subset of \((X \times Y)^*\).
  Notice that \(S\) is small since all quantifiers in its definition are with respect to fixed sets \(X\) and \(Y\)and \((X \times Y)^*\) is small (see Proposition~\ref{prop:family_set}).

  Let \(W \in \SCat{w}\) and let \(f \colon X \times Y \to W\) be a \(\Sigma\)-bilinear function.
  Define a subset \(Z \subseteq (X \times Y)^*\) as follows:
  \begin{equation*}
    Z = \{\fml{z} \in (X \times Y)^* \mid f\fml{z} \text{ is summable}\}
  \end{equation*}
  and define an equivalence relation \(\sim\) on \(Z\) where:
  \begin{equation*}
    \fml{z} \sim \fml{z'} \iff \Sigma f\fml{z} = \Sigma f\fml{z'}.
  \end{equation*}
  Define a function \(\bar{f} \colon Z/{\sim} \to W\) so that each equivalence class \([\fml{z}] \in Z/{\sim}\) is mapped to \(\Sigma f\fml{z}\); it is straightforward to check that \(\bar{f}\) is injective and, moreover, \([\emptyset] \in Z/{\sim}\), so the neutral element \(0 \in W\) is in the image of \(\bar{f}\).
  Therefore, according to Lemma~\ref{lem:wSm_restriction}, \(Z/{\sim}\) may be endowed with a partial function \(\Sigma^{\bar{f}}\) defined as follows for every family \(\{[\fml{z}_i]\}_I \in (Z/{\sim})^*\):
  \begin{equation*}
    \Sigma^{\bar{f}} \{[\fml{z}_i]\}_I = \begin{cases}
      [\fml{z}] &\ifc \exists [\fml{z}] \in Z/{\sim} \colon \Sigma \{\bar{f}[\fml{z}_i]\}_I \keq \bar{f}[\fml{z}] \\
      \undefined &\otherwise
    \end{cases}
  \end{equation*}
  so that \((Z/{\sim},\Sigma^{\bar{f}})\) is a weak \(\Sigma\)-monoid and \(\bar{f}\) is a \(\Sigma\)-homomorphism.
  Finally, define the function \(q \colon X \times Y \to Z/{\sim}\) as the composite
  \begin{equation*}
    q = X \times Y \xto{\{-\}} Z \xto{[-]} Z/{\sim}
  \end{equation*}
  where \(\{-\}\) is the function mapping each \((x,y) \in X \times Y\) to the singleton family \((\{(x,y)\})\) and \([-]\) is the quotient map.
  It is straightforward to check that \(f = \bar{f} \circ q\) and, for any \(x \in X\) and any summable family \(\{y_i\}_I \in Y^*\),
  \begin{equation*}
    \Sigma \{\bar{f} q(x,y_i)\}_I = \Sigma \{f(x,y_i)\}_I \keq f(x,\Sigma \{y_i\}_I) = \bar{f} q (x,\Sigma \{y_i\}_I)
  \end{equation*}
  since \(f\) is \(\Sigma\)-bilinear, implying that the family \(\{q(x,y_i)\}_I\) is summable in \((Z/{\sim},\Sigma^{\bar{f}})\) with:
  \begin{equation*}
    \Sigma^{\bar{f}} \{q(x,y_i)\}_I \keq q(x,\Sigma \{y_i\}_I).
  \end{equation*}
  A similar result holds if we fix \(y \in Y\) instead and let \(\fml{x} \in X^*\) be an arbitrary summable family, thus, \(q\) is a \(\Sigma\)-bilinear function.

  In summary, it has been shown that \(Z/{\sim}\) is an object in \(\SCat{w}\) and that \(X \times Y \xto{q} Z/{\sim}\) is a \(\Sigma\)-bilinear function, hence, an element of \(S\). Moreover, \(\bar{f}\) is a \(\Sigma\)-homomorphism and \(f = \bar{f} \circ q\), implying that \(\bar{f}\) is a valid morphism from \(q\) to \(f\) in in the comma category.
  This construction may be reproduced for any \(\Sigma\)-bilinear function \(f \colon X \times Y \to W\), so it follows that \(S\) is a weakly initial set.
  Thus, the comma category \(\comma{\{\bullet\}}{\SCat{w}^{X,Y}}\) has an initial object, as claimed.
\end{proof}

\begin{lemma} \label{lem:SCat_internal_hom}
  For objects \(X,Y \in \SCat{w}\), the hom-set \([X,Y]\) may be endowed with a partial function \(\Sigma^\to \colon [X,Y]^* \pto [X,Y]\) defined as:
  \begin{equation*}
    \Sigma^\to \fml{f} = \begin{cases}
      s_\fml{f} &\ifc s_\fml{f} \in [X,Y] \\
      \undefined &\otherwise
    \end{cases}
  \end{equation*}
  for every family \(\fml{f} = \{f_i\}_I \in [X,Y]^*\), where \(s_\fml{f}(x) = \Sigma \{f_i(x)\}_I\) for each \(x \in X\).
  Then, \(([X,Y], \Sigma^\to) \in \SCat{w}\).
\end{lemma} \begin{proof}
  Recall that all \(\Sigma\)-homomorphisms are total functions, so the condition \(s_\fml{f} \in [X,Y]\) imposes that \(\Sigma \fml{f}(x)\) is defined for all \(x \in X\).
  It is immediate that for every singleton family \(\{f\} \in [X,Y]^*\) the corresponding function \(s_{\{f\}}\) maps every \(x \in X\) to \(f(x)\), so \(\Sigma^\to \{f\} \keq f\) as required by the singleton axiom.
  Similarly, it is immediate that the function \(s_\emptyset\) maps every \(x \in X\) to the neutral element \(0 \in Y\), so the empty family is summable in \(([X,Y],\Sigma^\to)\), with the neutral element being \(s_\emptyset\).
  Moreover, let \(\fml{f} \in [X,Y]^*\) be a summable family and let \(\fml{f}_\emptyset\) be the subfamily where all occurrences of \(s_\emptyset\) have been removed; it is immediate that \(\fml{f}_\emptyset\) is summable since \(s_\emptyset\) only contributes to the sum of \(\fml{f}(x) \in Y^*\) by adding a \(0\), which may be disregarded thanks to the neutral element axiom in \(Y\).
  Consequently, the neutral element axiom is satisfied by \(\Sigma^\to\); it remains to prove the bracketing and flattening axioms.
  Let \(\{\fml{f}_j \in [X,Y]^*\}_J\) be a collection of summable families and let \(\fml{f} = \uplus_J \fml{f}_j\).
  Assume that \(\fml{f}_j\) is a finite family for every \(j \in J\) and assume that \(\fml{f}\) is summable; to prove bracketing we must show that the family \(\fml{g} = \{\Sigma^\to \fml{f}_j\}_J \in [X,Y]^*\) is summable and \(s_\fml{g} = s_\fml{f}\).
  This is straightforward since for every \(x \in X\) we have that:
  \begin{equation*}
    s_\fml{g}(x) = \Sigma \{(\Sigma^\to \fml{f}_j)(x)\}_J = \Sigma \{\Sigma \fml{f}_j(x)\}_J \keq \Sigma \fml{f}(x) = s_\fml{f}(x)
  \end{equation*}
  where the \(\keq\) step corresponds to bracketing in \(Y\).
  Therefore, \(s_\fml{g} \in [X,Y]\) is implied by \(s_\fml{f} \in [X,Y]\) --- which holds due to \(\fml{f}\) being summable --- and, hence, \(\fml{g} = \{\Sigma^\to \fml{f}_j\}_J\) is summable with its sum matching that of \(\fml{f}\).
  Flattening is proven via the same argument, this time using flattening in \(Y\) instead of bracketing.
  Consequently, if \(Y \in \SCat{w}\) then \([X,Y] \in \SCat{w}\).
\end{proof}

\begin{lemma} \label{lem:SCat_3_tensor}
  For arbitrary objects \(X,Y,Z,W \in \SCat{w}\), let \(f \colon (X \times Y) \times Z \to W\) be an arbitrary \(\Sigma\)-trilinear function.
  Then, there is a unique \(\Sigma\)-homomorphism \((X \otimes Y) \otimes Z \to W\) making the diagram
  \[\begin{tikzcd}
    {(X \times Y) \times Z} & W \\
    & {(X \otimes Y) \otimes Z}
    \arrow[dashed, from=2-2, to=1-2]
    \arrow["f", from=1-1, to=1-2]
    \arrow["{p\circ(p \times \id)}"', from=1-1, to=2-2]
  \end{tikzcd}\]
  commute in \(\Set\).
\end{lemma} \begin{proof}
  Let \(\mathrm{ev} \colon [Z,W] \times Z \to W\) be the function that maps each pair \((h,z)\) to \(h(z)\).
  Let \(h \in [Z,W]\) be an arbitrary \(\Sigma\)-homomorphism and \(\fml{z} = \{z_i\}_I \in Z^*\) an arbitrary summable family, then:
  \begin{equation*}
    \mathrm{ev}(h,\Sigma \fml{z}) = h(\Sigma \fml{z}) \keq \Sigma h\fml{z} = \Sigma \{\mathrm{ev}(h,z_i)\}_I.
  \end{equation*}
  Similarly, let \(\fml{h} = \{h_i\}_I \in [Z,W]^*\) be an arbitrary summable family and fix an arbitrary \(z \in Z\), then:
  \begin{equation*}
    \mathrm{ev}(\Sigma^\to \fml{h},z) = \mathrm{ev}(s_\fml{h},z) = s_\fml{h}(z) = \Sigma \{h_i(z)\}_I = \Sigma \{\mathrm{ev}(h_i,z)\}_I.
  \end{equation*}
  Therefore, \(\mathrm{ev}\) is a \(\Sigma\)-bilinear function.
  Let \(f \colon (X \times Y) \times Z \to W\) be a \(\Sigma\)-trilinear function.
  Since \(\Set\) is Cartesian closed, there is a unique function \(\Lambda{f} \colon X \times Y \to [Z,W]\) making the diagram
  \begin{equation} \label{diag:SCat_currying}
    \begin{tikzcd}
      {[Z,W] \times Z} \\
      {(X \times Y) \times Z} & W
      \arrow["f"', from=2-1, to=2-2]
      \arrow["{\Lambda{f} \times \id}", dashed, from=2-1, to=1-1]
      \arrow["{\mathrm{ev}}", from=1-1, to=2-2]
    \end{tikzcd}
  \end{equation}
  commute in \(\Set\).
  The function \(\Lambda{f}\) maps each pair \((x,y) \in X \times Y\) to the function \(f(x,y,-)\), which is a \(\Sigma\)-homomorphism since \(f\) is \(\Sigma\)-trilinear; we now show that \(\Lambda{f}\) is \(\Sigma\)-bilinear.
  Fix \(x \in X\) and let \(\fml{y} = \{y_i\}_I \in Y^*\) be an arbitrary summable family; then, for every \(z \in Z\):
  \begin{align*}
    \Sigma \{\Lambda{f}(x,y_i)(z)\}_I &= \Sigma \{f(x,y_i,z)\}_I  && \text{(definition of \(\Lambda{f}\))} \\
      &\keq f(x,\Sigma \fml{y},z)  && \text{(\(f\) is \(\Sigma\)-trilinear)} \\
      &= \Lambda{f}(x,\Sigma \fml{y})(z)  && \text{(definition of \(\Lambda{f}\))}
  \end{align*}
  with \(\Lambda{f}(x,\Sigma \fml{y})\) a \(\Sigma\)-homomorphism, as previously discussed.
  Therefore,
  \begin{equation*}
    \Sigma^\to \{\Lambda{f}(x,y_i)\}_I \keq \Lambda{f}(x,\Sigma \fml{y}).
  \end{equation*}
  The same can be said for summable families \(\fml{x} \in X^*\) if \(y \in Y\) is fixed instead, hence, \(\Lambda{f}\) is a \(\Sigma\)-bilinear function.

  Since both \(\Lambda{f}\) and \(\mathrm{ev}\) are a \(\Sigma\)-bilinear functions, the diagram
  \[\begin{tikzcd}[column sep=large]
    & {[Z,W] \times Z} \\
    {(X \otimes Y) \times Z} & {(X \times Y) \times Z} & W \\
    & {(X \otimes Y) \otimes Z}
    \arrow["f"{pos=0.4}, from=2-2, to=2-3]
    \arrow["{\Lambda{f} \times \id}"{pos=0.4}, dashed, from=2-2, to=1-2]
    \arrow["{\mathrm{ev}}", from=1-2, to=2-3]
    \arrow["{p \times \id}"', from=2-2, to=2-1]
    \arrow["{\overline{\Lambda{f}} \times \id}", dashed, from=2-1, to=1-2]
    \arrow["p"', from=2-1, to=3-2]
    \arrow[dashed, from=3-2, to=2-3]
  \end{tikzcd}\]
  commutes in \(\Set\) with each dashed arrow denoting uniqueness.
  The top right triangle is diagram~\eqref{diag:SCat_currying} which was already established to commute.
  Because \(\Lambda{f}\) is a \(\Sigma\)-bilinear function, Lemma~\ref{lem:SCatw_tensor} implies that the top left triangle commutes, where \(\overline{\Lambda{f}}\) is a \(\Sigma\)-homomorphism.
  Considering that \(\mathrm{ev}\) is a \(\Sigma\)-bilinear function, it is straightforward to check that \(\mathrm{ev} \circ (\overline{\Lambda{f}} \times \id)\) is also a \(\Sigma\)-bilinear function; then Lemma~\ref{lem:SCatw_tensor} implies that the outer edges of the diagram commute, with the corresponding unique function \((X \otimes Y) \otimes Z \to W\) being a \(\Sigma\)-homomorphism.
  Consequently, the triangle at the bottom of the diagram commutes.
  Notice that if we were to find another \(\Sigma\)-homomorphism \(g \colon (X \otimes Y) \otimes Z \to W\) making the bottom triangle commute, such \(g\) would also make the outer edges of the diagram commute but, according to Lemma~\ref{lem:SCatw_tensor}, there is only one such \(\Sigma\)-homomorphism. Therefore, the \(\Sigma\)-homomorphism \((X \otimes Y) \otimes Z \to W\) from the claim has been shown to exist and be unique, completing the proof.
\end{proof}

\subsection{Supporting proofs for Section~\ref{sec:HausCMon->SCatft}}
\label{sec:HausCMon->SCatft_appendix}

\begin{proposition}
  \(\HausCMon\) is a complete category.
\end{proposition} \begin{proof}
  It is well known that the category of commutative monoids and monoid homomorphisms \(\CMon\) is complete. It is also well known that the category of Hausdorff spaces and continuous functions \(\TopHaus\) is complete. Let \((X, \tau_X, +_X)\) and \((Y, \tau_Y, +_Y)\) be Hausdorff commutative monoids; we show that equalisers and products exist in \(\HausCMon\).

  Let \(f,g \in \HausCMon(X,Y)\) and let \(E = \{x \in X \mid f(x) = g(x)\}\). The equaliser of these morphisms in \(\TopHaus\) is \((E,\tau_E)\) where \(\tau_E\) is the subspace topology \(\tau_E = \{E \cap U \mid U \in \tau_X\}\). The equaliser of these morphisms in \(\CMon\) is \((E,+_E)\), where  and \(+_E\) is the restriction of \(+_X\) to domain \(E \times E\).
  Notice that \(+_E\) is continuous: it is the restriction of the continuous function \(+_X\) to the Hausdorff space \(E \times E\) with the subspace topology.
  Thus, \((E, \tau_E, +_E)\) is a Hausdorff commutative monoid.
  Moreover, considering that it is an equaliser in both \(\TopHaus\) and \(\CMon\), the corresponding universal morphisms are continuous monoid homomorphisms, and it immediately follows that \((E, \tau_E, +_E)\) is an equaliser in \(\HausCMon\).

  The product of \(X\) and \(Y\) in \(\TopHaus\) is \((X \times Y, \tau')\) where \(\tau'\) is the product topology. Their product in \(\CMon\) is \((X \times Y, +')\) where \(+'\) applies \(+_X\) to the \(X\) component and \(+_Y\) to the \(Y\) component.
  Notice that \(+'\) is continuous since it arises as the composition of continuous morphisms:
  \[\begin{tikzcd}
    {(X \times Y) \times (X \times Y)} && {X \times Y} \\
    {(X \times X) \times (Y \times Y)}
    \arrow["{+'}", from=1-1, to=1-3]
    \arrow["{+_X \times +_Y}"', from=2-1, to=1-3]
    \arrow["\cong"', from=1-1, to=2-1]
  \end{tikzcd}\]
  and, hence, \((X \times Y, \tau', +')\) is a Hausdorff commutative monoid. Moreover, the fact that it is a product in both \(\TopHaus\) and \(\CMon\) implies it is a product in \(\HausCMon\). The same argument can be used to prove that \(\HausCMon\) has all small products.

  Since \(\HausCMon\) has equalisers and small products, it is complete.
\end{proof}

\begin{proposition}
  The functor \(G \colon \HausCMon \to \SCat{ft}\) from Corollary~\ref{cor:HausCMon_ft} preserves limits.
\end{proposition} \begin{proof}
  First, we show that \(G\) preserves equalisers.
  Let \(X\) and \(Y\) be two Hausdorff commutative monoids and let \(f,g \in \HausCMon(X,Y)\); their equaliser \(E \in \HausCMon\) is defined in Proposition~\ref{prop:HausCMon_complete}, with inclusion morphism \(e \colon E \into X\).
  For every family \(\fml{x} \in E^*\) whose net of finite partial sums is denoted \(\sigma \colon \fpset{\fml{x}} \to E\) it follows that:
  \begin{equation*}
    \lim \sigma = x \iff \lim\, (e \circ \sigma) = e(x)
  \end{equation*}
  where we have used that \(E\) is endowed with the subspace topology.
  Since \(e\) is a monoid homomorphism it follows from the above and the definition of the extended monoid operation \(\Sigma^{G(E)}\) of \(G(E)\) that:
  \begin{equation*}
    \Sigma^{G(E)} \fml{x} \keq x \iff \Sigma^{G(X)} e\fml{x} \keq e(x)
  \end{equation*}
  where \(\Sigma^{G(X)}\) is the extended monoid operation of \(G(X)\).
  The latter two-way implication uniquely characterises the \(\Sigma\) function~\eqref{eq:SCat_equaliser} of the equaliser in \(\SCat{ft}\) and, hence, \(G\) preserves equalisers.

  To prove that \(G\) preserves products we follow a similar strategy.
  Let \(\fml{p} \in (X \times Y)^*\) and let \(\sigma\) be its net of finite partial sums. Since the product \(X \times Y\) in \(\HausCMon\) is assigned the product topology, it follows that:
  \begin{equation*}
    \lim\, \sigma = (x, y) \ \iff \ \lim\, (\pi_X \circ \sigma) = x \text{ and } \lim\, (\pi_Y \circ \sigma) = y.
  \end{equation*}
  Since the monoid operation in \(X \times Y\) is defined coordinate-wise, it follows that extended monoid operation \(\Sigma^{G(X \times Y)}\) of \(G(X \times Y)\) satisfies:
  \begin{equation*}
    \Sigma^{G(X \times Y)} \fml{p} \keq (x, y) \ \iff \ \Sigma^{G(X)} \pi_X\fml{p} \keq x \text{ and } \Sigma^{G(Y)} \pi_Y\fml{p} \keq y
  \end{equation*}
  where \(\Sigma^{G(X)}\) and \(\Sigma^{G(Y)}\) are the extended monoid operation assigned to each \(G(X)\) and \(G(Y)\) respectively.
  The latter two-way implication uniquely characterises the \(\Sigma\) function~\eqref{eq:SCat_product} of the categorical product in \(\SCat{ft}\) and, hence, \(G\) preserves products. This argument can be generalised to arbitrary small products.

  Since \(G\) preserves equalisers and small products and \(\HausCMon\) is complete, \(G\) preserves limits.
\end{proof}

\end{document}